\documentclass[12pt,a4paper]{article}

\usepackage{amsfonts, amsmath, amssymb, amsthm,url,dsfont, gensymb,multicol}
\usepackage[margin=3cm,footskip=1cm]{geometry}

\usepackage{graphicx,xcolor}
\usepackage[english]{babel}
\usepackage{hyperref}
\usepackage{booktabs}
\usepackage{tikz}
\usepackage{array}
\usepackage{cellspace}
\usepackage[round]{natbib}
\usepackage[above,below,section]{placeins}

\usepackage{etoolbox}
\AtBeginEnvironment{thebibliography}{%
  \small
  \setlength\itemsep{0.75em plus 0.5em minus 0.75em}%
}

\usepackage{caption}
\usepackage{authblk}
\usepackage{xspace}

\makeatletter
\newcommand\CorrespondingAuthor[1]{%
  \begingroup%
  \def\@makefnmark{}%
  \footnotetext{Corresponding author: #1}%
  \endgroup%
}
\makeatother

\makeatletter
\renewenvironment{abstract}{%
  \small%
  \providecommand\keywords{%
    \par\medskip\noindent\textit{Keywords:}\xspace}%
  \begin{center}%
    \bfseries \abstractname\vspace{-.5em}\vspace{\z@}%
  \end{center}%
  \quote%
}{\endquote}
\makeatother

\usepackage[shortlabels]{enumitem}
\setlist{
  listparindent=\parindent,
  parsep=0pt,
}

\newtheorem{thm}{Theorem}[section]
\newtheorem{cor}[thm]{Corollary}
\newtheorem{lem}[thm]{Lemma}
\newtheorem{defi}[thm]{Definition}
\newtheorem{prop}[thm]{Proposition}
\theoremstyle{definition}

\theoremstyle{remark} 

\newtheorem{rem}[thm]{Remark}

\def\mc{\mathcal}

\DeclareMathOperator{\Cov}{Cov}
\DeclareMathOperator{\Var}{\mathsf{Var}}

\newcommand{\X}{\mathcal{X}}

\newcommand{\E}{\mathbb{E}}
\def\P{\mathbb P}
\newcommand{\Ha}{\mathcal{H}}
\newcommand{\PP}{\mathcal{P}}
\def\PPn{\mc P_n}
\newcommand{\N}{\mathbb{N}}
\newcommand{\Z}{\mathbb{Z}}
\newcommand{\R}{\mathbb{R}}

\newcommand{\eps}{\varepsilon}

\def\d{{\mathrm{d}}}

\def\xx{\boldsymbol x}
\def\dist{\mathsf{dist}}

\def\bR{\breve{\mathbb R}}
\def\bWn{\breve W_n}
\def\bxx{\breve{\xx}}
\def\bx{\breve x}
\def\bPPn{\breve{\mc P_n}}

\def\doublespacing{\renewcommand{\baselinestretch}{1.5}\large\normalsize}
\begin{document}
\doublespacing

\title{A functional central limit theorem for the empirical
Ripley's $K$-function}	
\author[1]{Christophe Ange Napol\'eon Biscio}
\author[1]{Anne Marie Svane}
\affil[1]{Department of Mathematical Sciences,
Aalborg University, Skjernvej 4A, DK-9220 Aalborg, Denmark}

\date{}
\maketitle

\begin{abstract}
	We establish a functional central limit theorem for Ripley's $K$-function for two classes of point processes. One is the class of point processes having exponential
	decay of correlations and further 
	satisfying a conditional m-dependence condition. The other is a
	family of Gibbs point processes. We illustrate the use of our
	theorem for goodness-of-fit tests in simulations.


	\keywords point processes, Ripley's $K$-function, functional
	central limit theorem, goodness-of-fit test 
\end{abstract}

\section{Introduction}	

Ripley's $K$-function \citep{ripley2} is a classical way of
summarizing the second order structure of a stationary point
process in spatial statistics. For a given $r>0$, $K(r)$
is defined to be the expected number of points in a ball of
radius $r$ around a typical point of the point process. Apart
from being geometrically intuitive, an important reason for
studying the $K$-function is that for isotropic point processes
it contains all information about the pair correlation function,
see e.g.~\cite{chiu}.
Consequently, the Ripley's $K$ function is commonly used for
statistical inference in spatial statistics for point processes,
see e.g.~\cite{Baddeley:Rubak:Wolf:15}.

A common problem in this field is to assess how well an observed
point pattern $\xx$ fits an assumed model $\PP_0$ based on a
chosen summary statistics, e.g. the Ripley's $K$ function
and the pair correlation function. A standard way to perform
goodness-of-fit tests is to use confidence bands derived
from central limit theorems of  the chosen summary
statistic under $\PP_0$ . However, in the case of the Ripley's $K$ function,
such theorems have only been established when $\PP_0$ is a
Poisson point process with known or unknown intensity,
see~\cite{heinrich91}. When no confidence bands are available,
pointwise envelope tests, see e.g.~\cite{Baddeley:Rubak:Wolf:15},
and their extension to global envelope tests have been a popular
alternative~\citep{globalenv17}. However, pointwise envelope
tests are restricted to comparison of the summary statistics
evaluated at a given argument while global envelope tests require
several thousands of simulations of point patterns following
$\PP_0$ with on average the same number of points as in the
observed pattern $\xx$. This allows for goodness-of-fit test of
more models than the Poisson point process but may be infeasible
when the number of observed points is large.

Thus, establishing central limit theorems for the  $K$-function
of $\PP_0$ when it is not necessarily a Poisson point process
would allow for goodness-of-fit tests which are not based on
envelope methods as in~\cite{heinrich91}. In
Theorem~\ref{mainthm} of the present paper, we provide a first
necessary step in that direction by establishing central limit
theorems for Ripley's $K$-function when $\PP_0$ does not depend
of unknown parameters. However, a problem arises when  $\PP_0$
belongs to a parametric family of models that contains several
unknown parameters, e.g. the intensity as described
in~\cite{Baddeley:Rubak:Wolf:15}. As observed in
\cite{heinrich91} for Poisson processes, replacing the intensity
by an estimate changes the limiting distribution. For more
general point process models, generalising Theorem~\ref{mainthm}
in the presence of unknown parameters appears to be challenging
and require a further development on our main result
in Theorem~\ref{mainthm}.  We thus leave this open for future
works.
To the best of the authors' knowledge, the only previous
contribution to a functional central limit theorem for the
$K$-function is~\cite{heinrich91,kclt} where the framework is
limited to stationary Poisson processes. Other recent results on
functional central limit theorems for geometric summary
statistics are given in \cite{hirsch,owada21}.

The key to our results is to rewrite the estimator for Ripley's $K$-function  in the form
\begin{equation*}
\sum_{x\in \PP} \xi (x,\PP \cap W_n),
\end{equation*}
where $\PP$ is the point process, $W_n$ is an observation window,
and $\xi$ is a so-called score function. There has been a lot of
recent activity to establish the asymptotic behaviour of  summary
statistics of this form \citep{yogesh,gibbsCLT} when either  $\PP$
has fast decay of correlations or it belongs to a suitable class
of Gibbs processes. This immediately leads to point-wise laws of
large numbers for the mean and variance and, when a certain
variance lower bound is satisfied, these papers also provide
point-wise central limit theorems. This variance bound can be
shown for the $K$-function by applying techniques of
\cite{bchs20,gibbsCLT}.  To obtain a functional central limit
theorem, tightness is shown by applying a machinery developed in
\cite{bchs20} for persistence diagrams and adapting it to Gibbs
point processes.

The paper is structured as follows. In Section \ref{sec:K} we
introduce the $K$-function and its most common edge corrected estimators. In Section \ref{sec:classes} we introduce the
classes of point processes to which our results apply. The main
results for the $K$-function are stated in Section
\ref{sec:mainresults} and corollaries for related functionals are
given in Section \ref{sec:corollaries}. We investigate the
statistical performance of a test based on the central limit
theorem in Section \ref{sec:gof} before proving the main results
in Sections~\ref{sec:proofs}-\ref{s:tightsec}.
Appendix~\ref{app:condvar} contains two required
results on the conditional variance of random variables and
Appendix~\ref{app:thinning} presents some background on Gibbs point
processes. 

\section{The $K$-function} \label{sec:K}

Let $\PP\subseteq \R^d$ be a simple stationary point process of intensity
$\rho>0$ and $A\subseteq \R^d$ a set of  positive and finite volume $|A|$. 
For all $r\geq 0$,  the $K$-function is defined by
\begin{equation*}
  K(r) = \frac{1}{|A|\rho^{2}}\E \sum_{x\in \PP\cap A}
  \sum_{y\in \PP }
  \mathds{1}_{\{0<|x-y|\leq r\}} 
  = \frac{1}{\rho}\E_o
  \sum_{y\in \PP}
  \mathds{1}_{\{0<|y|\leq r\}},
\end{equation*}
where $\E_o$ denotes the Palm expectation given $o\in \PP$. This
definition is independent of the set $A$.

Typically, the point process is only observed inside a bounded
observation window. Throughout this paper, we will consider a
square observation window  $W_n =
[-\frac{1}{2}n^{1/d},\frac{1}{2}n^{1/d}]^d$ of volume $n$ and
write $\PP_n = \PP \cap W_n$. The most naive estimator for the
$K$-function based on $\PP_n$ is  
\begin{equation}\label{eq:def estim K no correction}
  \hat{K}_n(r) = \frac{1}{n\rho^{2}} \sum_{x\in \PP_n}
\sum_{y\in \PP_n}
\mathds{1}_{\{0<|x-y|\leq r\}}.
\end{equation}
However, this estimator is downward biased due to points $y$
close to the edges of $W_n$ not being counted. This bias tends to
zero when the volume of the window goes to infinity, see
Theorem \ref{LLNthm} below, but for finite window sizes, an edge
corrected estimator is typically used to avoid the bias
\citep{ripley}. This is a weighted estimator of the form 
\begin{equation}\label{eq:def estim K}
  \hat{K}_{e,n}(r) = \frac{1}{n\rho^{2}} \sum_{x\in \PP_n}
  \sum_{y\in \PP_n }
  \mathds{1}_{\{0<|x-y|\leq r\}} e_{n}(x,y),
\end{equation}
where $e_{n}$ is an edge correction factor depending on the window.  Several
edge correction factors have been proposed in the literature, see e.g.\ \cite{ripley}. In the
case of stationary point processes, the most commonly used edge corrections are:
\begin{itemize}
\item No correction: $e_{1,n}(x,y) = 1$. This corresponds to the
uncorrected estimator $\hat{K}_n$.
\item Translation correction:
$e_{2,n}(x,y) = \frac{|W_{n}|}{|W_{n} \cap (W_{n}+x-y)|}$, where
$|\cdot| $ denotes volume.
\item Rigid motion correction:  
\begin{equation*}
e_{3,n}(x,y) = \frac{|W_{n}|}{\int_{SO(d)}|W_{n} \cap (W_{n}+\eta (x- y))|\nu(\d \eta)},
\end{equation*}
 where $SO(d)$ is the space of rotations equipped with the
 normalized Haar measure $\nu$. This is the inverse of the
 proportion of all rigid motions keeping $x$ in $W_n$ that also
 keep $y$ in $W_n$.
\item Border correction (minus sampling): $e_{4,n}(x,y) = \mathds{1}_{W_n \ominus B_r(0)}(x)\frac{n}{|W_n\ominus B_r(0)|}$, where $B_r(x)$ is the ball around $x$ of radius $r$ and $\ominus$ denotes Minkowski set difference. This edge correction is equivalent to sampling $x$ from a smaller window such that all points within distance $r$ from $x$ can be observed in $W_n$.
\item Isotropic correction: $e_{5,n}(x,y) = \frac{\Ha^{d-1}(B_{|x-y|}(x))}{\Ha^{d-1}(B_{|x-y|}(x)\cap W_n)}$, where $\Ha^{d-1} $ denotes $(d-1)$-dimensional Hausdorff measure (surface area).
\end{itemize}
The  edge corrections $e_{2,n}$ and $e_{4,n}$ lead to unbiased estimators for all stationary point processes, while $e_{3,n}$ and $e_{5,n}$ also require that the point process is  isotropic to yield unbiasedness. 

In the proofs below, we write the estimators  as 
\begin{align*}
  \hat{K}_{n}(r){}&  =  \frac{1}{n} \sum_{x\in \PP_n} \xi_{r}(x,\PP_n)\\
  \hat{K}_{e,n}(r) {}& =  \frac{1}{n} \sum_{x\in \PP_n} \xi_{e,n,r}(x,\PP_n),
\end{align*}
where $\xi_r$ and $\xi_{e,n,r}$ are the so-called score functions defined for a locally finite point pattern $\mathcal{X}$ and a point $x\in \mathcal{X}$ as
\begin{align}\label{score}
  \xi_{r}(x,\mathcal{X}) {}&=  \frac{1}{\rho^{2}} \sum_{y\in \mathcal{X} } \mathds{1}_{\{0<|x-y|\leq r\}}\\ \nonumber
\xi_{e,n,r}(x,\mathcal{X}) {}&=  \frac{1}{\rho^{2}} \sum_{y\in \mathcal{X} } \mathds{1}_{\{0<|x-y|\leq r\}} e_{n}(x,y). 
\end{align}
Note that, for all the non-trivial edge corrections listed above, the associated score functions depend on $n$. Moreover, some of them are not invariant with respect to translations $(x,\mathcal{X} )\mapsto (x+y,\mathcal{X} + y)$. 


\section{Classes of point processes} \label{sec:classes}
In this section, we introduce the two main types of point processes that we are going to consider. One is the class of conditionally $m$-dependent point processes 
 having fast decay of
correlations as considered in \cite{bchs20}. The second is a class of Gibbs point processes considered in \cite{gibbs_limit}. In Section \ref{sec:LLN}, we state a formula from \cite{yogesh,gibbsCLT} for the limiting mean and covariance of $\hat{K}_n(r)$ for fixed value(s) of $r$ when the volume of the observation window goes to infinity.

A point process $\PP$ is formally defined as a random variable taking values in
the space of locally finite subsets $\mathcal{N}$ of $\R^d$ endowed with the
smallest $\sigma$-algebra $\mathfrak N$ such that the number of points
in any given Borel set is measurable. We assume that $\PP$ is simple with intensity $\rho$ and stationary.

We will assume that all factorial moment measures exist and are
absolutely continuous, that is, the \emph{$p$th factorial moment
density} $\rho^{(p)}$ is determined via the identity
\begin{align*}
\E\Big[\prod_{i \le p} \PP(A_i)\Big] = \int_{A_1 \times \cdots \times A_p} \rho^{(p)}(x_1,\ldots,x_p)\d x_1\dotsm \d x_p 
\end{align*}
for all pairwise disjoint bounded Borel sets $A_1, \dots, A_p \subseteq \R^d$. Here $\PP(A_i)$ denotes the number of points of $\PP$ in $A_i$. 

\subsection{Conditionally $m$-dependent point processes}
The first class of point processes we will consider satisfy a set
of conditions that we introduce in this section, namely fast
decay of correlations, conditional $m$-dependence, and Conditions
{\bf (M)} and {\bf (R)} below. We say that a function $\phi: [0,
\infty) \to [0, 1]$  is \emph{fast decreasing} if $\lim_{t \to
\infty} t^m \phi(t) = 0$ for all $m \ge 1$.

\begin{defi}\label{def:expdecay} Let $\PP$ be a simple stationary
  point process in $\R^d$, such that the $p$th factorial moment density
  $\rho^{(p)}$ exist for all $p \ge 1$.  Then, $\PP$ exhibits
  \emph{fast decay of correlations} if there exists a fast
  decreasing function $\phi$ and 
    constants $c_n,C_n>0$ for all $n\in \mathbb{N}$ such that for any $p,q\in \mathbb{N}$ and
  $\xx = \{x_1, \dots, x_p\}, \xx' = \{x_{p + 1}, \dots, x_{p + q}\}
  \subset \R^d$,
  \begin{equation*}
    |\rho^{(p + q)}(\xx,\xx') - \rho^{(p)}(\xx)
    \rho^{(q)}(\xx')| \leq C_{p + q} \phi(c_{p+q}\, \dist(\xx, \xx')).
    \end{equation*}
    Here $\dist$ denotes the distance between two point sets, i.e.\
    \begin{equation*}
    \dist(\xx,\xx')= \min_{x_i\in \xx,x_j\in \xx'} |x_i-x_j|.
    \end{equation*} 
\end{defi}

For $p\in \N$, let $\PP^{\ne}_p$ denotes $p$-tuples of
pairwise distinct points in $\PP$ and recall that the
\emph{$p$-point Palm distribution} $\P_{x_1,\dots,x_p}$ is
determined by
\begin{align*}
\E\Big[{}&\sum_{{(x_1, \dots, x_p) \in \PP^{\ne}_p}} f(x_1, \dots, x_p; \PP) \Big]\\
& = \int_{\R^{pd}}\E_{x_1,\ldots,x_p}[f(x_1,\ldots,x_p; \PP)]  \rho^{(p)}(x_1,\ldots,x_p ) \d x_1,\ldots,x_p,
\end{align*}
for any bounded measurable $f: \R^{pd} \times \mathcal{N} \to \R
$. 
 With this notation, we make the following
condition on the Palm moments: 
\begin{itemize}
	\item[{\bf (M)}] For every $p\ge1$,
	$$\sup_{\substack{l \le p \\ x_1,\ldots,x_l \in \R^{d}}}\E_{x_1,\ldots,x_l}[\PP(W_1)^p] < \infty.$$
\end{itemize}

We summarize for later reference the conditions of \cite{yogesh} that are satisfied by the point processes and score functions we consider.
\begin{lem}\label{A2} Suppose that $\PP$ has fast decay of
	correlations and let $\xi$ be a linear combination of score
	functions $\xi_{r_1},\ldots,\xi_{r_p}$ of the form
	\eqref{score}. Then $(\PP,\xi)$ is admissible of class (A1)
	in the sense of \cite{yogesh} and $\xi(x,\mathcal{X})$
	depends only on $\mathcal{X}\cap B_r(x)$ where
	$r=\max_{i=1,\ldots,p} r_i$. If, moreover, $\PP$ satisfies
	the moment condition {\bf (M)},  then the $p$-moment
	condition  \cite[(1.19)]{yogesh} is satisfied, i.e.
	for all $p> 0$, there is an $M_p>0$ such that
	\begin{equation*}
    	\sup_{1\leq n\leq \infty} \sup_{1\leq l\leq  p  } \sup_{x_1,\ldots,x_{l} \in \R^d} \E_{x_1,\ldots,x_{l}} (|\xi_r(x_1,\PP_n)|\vee 1)^p \leq M_p.
	\end{equation*}
\end{lem}

\begin{proof}
	It suffices to show the lemma for $\xi=\xi_r$, the extension
	to general linear combinations being trivial. 
	
	Since $\PP$ has fast decay of correlations by assumption and
	since $\xi_r(x,\PP)$ has the form of a $U$-statistics, i.e.\
	\begin{equation*}
	\xi_r(x,\mathcal{X}) =\tfrac{1}{2}\sum_{y\in \mathcal{X}} h(x,y)
	\end{equation*}
	where $h(x,y)=\frac{2}{\rho^2}\mathds{1}_{\{0<|x-y|\leq r\}}$
	is bounded and translation invariant and vanishes for
	$|x-y|>r$, the pair $(\PP,\xi)$ is of class (A1). 

    Since
    \begin{equation*}
   		\E_{x_1,\ldots,x_{l}} (|\xi_r(x_1,\PP_n)|\vee 1)^p \leq \E_{x_1,\ldots,x_{l}} \PP(B_r(x_1))^p, 
    \end{equation*}  
    the $p$-moment condition
	\cite[(1.19)]{yogesh} follows from the moment condition {\bf
	(M)}.
\end{proof}

The central limit theorems in Section \ref{sec:mainresults}, require a variance lower bound. To obtain this, we make two further assumptions. For this, we assume that we study the $K$-function on an interval $[r_0,R]$.

\begin{defi} A point process $\PP$ is said to be \emph{conditionally $m$-dependent} if there exists a random measure $\Lambda$ such that
$\PP \cap A$ and $\PP \cap A'$ are conditionally independent given
$\sigma(\Lambda, \PP \cap {A''})$ for any bounded Borel sets
$A, A', A'' \subseteq \R^d$ such that the distance between $A$ and $A'$
is larger than some $m$. Here, $\sigma(\Lambda, \PP \cap {A''})$
denotes the $\sigma$-algebra generated by $\Lambda$ and
$\PP \cap {A''}$.  For such a process, we let $\tilde{R} = \max(m,R)$. 
\end{defi}

For a conditionally $m$-dependent process, we introduce the following condition, which essentially guarantees the realizability of certain point configurations. 

\begin{enumerate}
	\item[{\bf (R)}] Let $\PP$ be a conditionally $m$-dependent point process. For any $r_0\leq r_1<r_2 \leq R$, define events $F_1, F_2 \in \mathfrak N$ by
	\begin{align*}
	F_{1} ={}& \lbrace  \forall (x,y) \in (\PP_{(5\tilde{R})^{d}})_2^{\neq} : |x-y|>R \rbrace \\
	F_{2} ={}& \left\lbrace \forall (x,y) \in (\PP_{(5\tilde{R})^{d}})_2^{\neq}: |x-y|>r_1\right\rbrace \\
	&\cap \left\lbrace  \exists (x,y) \in (\PP_{(3\tilde{R})^{d}})_2^{\neq}: |x-y|\leq r_2 \right\rbrace.
	\end{align*}
	Then it must hold that
	\begin{equation*}
	\E\big[\min_{i \in \{1,2\}} \P\big( \PP_{(5\tilde{R})^d} \in F_i\,|\, \sigma(\Lambda , \PP \setminus W_{\tilde{R}^d})\big)  \big] > 0,
	\end{equation*}
	where $\Lambda$ is the measure from the definition of conditional $m$-dependence.
\end{enumerate}

Examples of conditionally $m$-dependent processes having fast
decay of correlations and satisfying {\bf (M)} and {\bf (R)} are
log-Gaussian Cox processes and Mat\'{e}rn cluster processes, see
\cite{bchs20}. Note that, although the proof for
log-Gaussian Cox processes is only given for compactly supported
covariance functions, the arguments easily generalize to fast
decaying covariance functions such as the exponential covariance
function, which will be used for simulations in Section
\ref{sec:gof}.

\subsection{Gibbs point processes}\label{sec:gibbs}
The second class of point processes we shall consider is a class of Gibbs point processes which we call $\mathbf{\Psi}^*$. This will be almost the same as the class $\mathbf{\Psi}^*$ in \cite{gibbs_limit,gibbsCLT} but with a few restrictions. We consider an energy functional $\Psi$ defined on finite point sets $\X$ that  satisfies the following conditions:
\begin{itemize}
	\item Translation invariance: $\Psi(\X) = \Psi(\X + x)$ for all $x\in \R^d$.
	\item Rotation invariance: $\Psi(\X) = \Psi(\X')$ whenever $\X'$ is a rotation of $\X$.
	\item Monotonicity: $\Psi(\X) \leq \Psi(\X')$ whenever $\X\subseteq \X'$.
	\item Positivity: $\Psi(\X)\in [0,\infty]$.
	\item Non-degeneracy: $\Psi(\{x\})<\infty$ for all $x\in \R^d$.
\end{itemize}

Let 
\begin{equation*}
\Delta^{\Psi} (x,\X) = \Psi(\X\cup \{x\}) - \Psi(\X)
\end{equation*}
with the convention $\infty - \infty = 0$. We say that $\Psi $ has finite range if there is a radius $r^{\Psi}$ such that
\begin{equation*}
\Delta^\Psi(x,\X) = \Delta^\Psi(x,\X\cap B_{r^{\Psi}}(x)).
\end{equation*}
for all $(x,\mathcal{X})$.

For a  finite range energy functional $\Psi$, we may  consider the infinite volume Gibbs point process with inverse temperature $\beta$ and activity $\tau$ satisfying 
\begin{equation}\label{eq:tau_condition}
\tau \kappa_d (r^{\Psi} )^d<1,
\end{equation}
 where $\kappa_d =\pi^{d/2}/\Gamma(1+d/2)$ is the volume of the $d$-dimensional unit ball.
Condition \eqref{eq:tau_condition}, together with the finite
range, ensures existence and uniqueness of the infinite volume
Gibbs process \cite[Thm. 4]{dereudre}. This  is the point process
$\PP$ satisfying that for any bounded domain $D$,
conditionally on $\PP\cap D^c=\X_0$,
$\PP\cap D$ is absolutely continuous with respect to a
Poisson process $\mathcal{Q}$  on $ D$ of intensity
$\tau$ with density 
\begin{equation}\label{eq:gibbs_def}
\X \mapsto 
\frac{\exp(-\beta(\Delta^\Psi_D(\X , \X_0 ))}{ 
	\E[ \exp(-\beta\Delta^\Psi_D(\mathcal{Q}  , \X_0))]},
\end{equation}
where expectation is taken with respect to $Q$, and for $\X
\subseteq D$, 
\begin{equation*}
\Delta^\Psi_D(\X , \X_0 ) = 
\Psi(\X \cup (\X_0 \cap (D\oplus B_{r^\Psi}(0)))) - 
\Psi(\X_0 \cap (D\oplus B_{r^\Psi}(0))).
\end{equation*}

The class $\mathbf{\Psi}^*$ consists of all infinite volume Gibbs point processes satisfying \eqref{eq:tau_condition}, where the energy functional has one of the following forms: 
\begin{itemize}
	\item[(i)] Pair potential: There is a pair potential function $\phi:[0,\infty)\to [0, \infty]$ such that $\phi$ has compact support, $\phi^{-1}(\infty)\subseteq [0,r_0]$ and $\phi$ is bounded on compact subintervals of $(r_0,\infty )$. Then
	\begin{equation*}
	\Psi(\{x_1,\ldots,x_n\}) = \sum_{i<j} \phi(|x_i-x_j|).
	\end{equation*}
	\item[(ii)] Area interaction process: Let $K\subseteq  \R^d $ be a deterministic compact convex set. Then
	\begin{equation*}
	\Psi(\{x_1,\ldots,x_n\}) = \Big|\bigcup_{i=1}^n(x_i+K)\Big|  
	\end{equation*}
	\item[(iii)] For a fixed $R>0$ and $k>2$,
	\begin{equation*}
	\Psi(\{x_1,\ldots,x_n\}) =\infty
	\end{equation*}
	if there is a ball of radius $R$ containing at least $k$ of the points. Otherwise, $\Psi(\{x_1,\ldots,x_n\}) =0$. 
\end{itemize}
Note that all these energy functionals have finite range. It was shown in \cite{gibbs_limit} that the point processes in $\PP$ can be constructed by a backwards oriented  perfect simulation technique, which is recalled in Appendix \ref{app:thinning} for reference in the proofs. It follows from this construction that all Gibbs point processes of class $\mathbf{\Psi}^*$ have fast decay of correlations as noted in \cite{yogesh}.

\begin{rem}
	The energy functionals here are essentially the same as in \cite{gibbsCLT,gibbs_limit}, except that 1) we modified the pair potentials in (i) to have finite range and 2) some of the energy functionals in \cite{gibbs_limit} allowed an extra term $\alpha n$, where $n$ is the number of points in $\X$. However, removing the term $\alpha n$ yields the same point process if $\tau$ is replaced by $\tilde{\tau}=\tau \exp(-\beta \alpha)$, see \cite[Def. 8]{dereudre}. The requirement
	\begin{equation*}
	\tau \exp(-\beta \alpha) \kappa_d (r^{\Psi} )^d<1
	\end{equation*}
	of \cite{gibbsCLT,gibbs_limit} then becomes equivalent to our
	condition $\tilde{\tau} \kappa_d (r^{\Psi} )^d<1.$
\end{rem}

\subsection{Laws of large numbers}\label{sec:LLN}
    
For both types of point processes we consider, the literature provides formulas for the limiting mean and variance when $n\to \infty$.

\begin{thm}[LLN for $\hat{K}_n(r)$]\label{LLNthm} Let $\PP$ be
	either a simple stationary point process in $\R^{d}$ that
	exhibits fast decay of correlations as in
	Definition~\ref{def:expdecay} and satisfies Condition {\bf
	(M)} or a Gibbs process of class $\mathbf{\Psi}^*$. Further,
	let $\hat{K}_n$ be defined as in~\eqref{eq:def estim K no
	correction}. Then, for any $r>0$, there is a constant $C_r$
	such that
	\begin{equation}\label{asympbias}
		\left| \E\hat{K}_n(r) -
		K(r)  \right|
		\leq C_rn ^{-\frac{1}{d}}.
	\end{equation}
	Moreover, for $r_1,r_2>0$,
	\begin{equation}\label{VarLim}
		\lim_{n \rightarrow \infty} 
		n \Cov (\hat{K}_n(r_1),\hat{K}_n(r_2)) =
		\frac{1}{\rho^{3}} \E_{o}  \sum_{x,y \in\PP\setminus \lbrace 0 \rbrace} \mathds{1}_{\{|x|\leq r_1\}} \mathds{1}_{\{|y|\leq r_2\}} +
		\int_{\R^{d}} a_{r_1,r_2}(x) \d x
	\end{equation}
	where
	\begin{align*}
	a_{r_1,r_2}(x) =  {}& \frac{\rho^{(2)}(0,x)}{2\rho^{4}} \E_{{0,x}} \Big( \sum_{u\in \PP\setminus \lbrace 0\rbrace }\sum_{v\in \PP\setminus \lbrace x \rbrace }  \Big(\mathds{1}_{\{|u|\leq r_1\}}
	\mathds{1}_{\{|x-v|\leq r_2\}} +\mathds{1}_{\{|u|\leq r_2\}}
	\mathds{1}_{\{|x-v|\leq r_1\}} \Big) \Big) \\
	&   -   K(r_1)K(r_2).
	\end{align*}
	The limit \eqref{VarLim} is finite.
	In particular, $\hat{K}_n(r)$ converges in probability towards $K(r)$. 
\end{thm}

\begin{proof}
	In the case of fast decay of correlations, the inequality
	\eqref{asympbias} and the limit \eqref{VarLim} for $r_1=r_2$
	follow directly from \cite[Thm. 1.12]{yogesh}  and Lemma
	\ref{A2}. The case $r_1\neq r_2$ in \eqref{VarLim} follows
	from the identity
	\begin{equation*}
	\Cov(X,Y)=\tfrac{1}{2}(\Var(X+Y)-\Var(X)-\Var(Y))
	\end{equation*}
{together with \cite[Thm. 1.12]{yogesh}
	and Lemma~\ref{A2} applied to $\xi_{r_1} + \xi_{r_2}$.}
	
	For Gibbs point processes, the first statement follows from a direct computation valid for any stationary point process, while the limiting covariance follows from \cite[Thm. 1.1]{gibbsCLT}. 
\end{proof}


\section{Main results}\label{sec:mainresults}

In this section, we state the main results of this paper, which
is a  central limit theorem for $\hat{K}_{e,n}(r)$ when
restricted to a bounded interval $[r_0,R]$. Throughout the paper,
$r_0$ and $R$ will denote the constants in condition \textbf{(R)}
in the case of conditionally $m$-dependent processes, and the
constants from the definition of the energy functional for Gibbs
point processes. 

We first state a central limit theorem for the finite dimensional
distributions of $\hat{K}_{e,n}(r)$ as described
in~\eqref{eq:def estim K}.  The proof is given in Section
\ref{sec:proofs} . 
\begin{thm}\label{thmfindim} 
	Let $\PP$ be a conditionally
	$m$-dependent point process having fast decay of correlations
	and satisfying condition {\bf (M)} and {\bf (R)}  or a Gibbs
	point process having energy functional of class
	$\mathbf{\Psi}^*$. Let $r_0\leq r_1<\dotsm < r_p \leq R$ and
	let $e=e_{i,n}, i\in \{1,\ldots,5\}$ be one of the
	edge corrections from Section \ref{sec:K}. Then 
	\begin{equation*}
	\sqrt{n}\Big(\hat{K}_{e,n}(r_1)- \E\hat{K}_{e,n}(r_1),\ldots, \hat{K}_{e,n}(r_p)-\E\hat{K}_{e,n}(r_p)\Big)
	\end{equation*}
	converges in distribution to a multivariate Gaussian variable
	with mean zero and covariance structure given by Theorem
	\ref{LLNthm}. 
\end{thm}

The next theorem is a functional central limit theorem for $\hat{K}_{e,n}(r)$. The proof is given in Section \ref{s:tightsec}
\begin{thm}\label{mainthm}
	Let $\PP$ be a conditionally $m$-dependent point process  with fast decay of
	 correlations satisfying conditions {\bf (M)} and {\bf (R)} or a Gibbs point process of class $\mathbf{\Psi}^*$.  Let $e=e_{i,n}$, $i\in \{1,\ldots,5\}$, be one of the edge corrections in Section \ref{sec:K}.
	The process
	\begin{equation*}
	\big\{\sqrt{n}(\hat{K}_{e,n}(r)-\E\hat{K}_{e,n}(r)) \big\}_{r\in [r_0,R]} 
	\end{equation*}
	converges weakly in Skorokhod topology to a centered Gaussian process with covariance structure given by Lemma \ref{LLNthm}.  The limiting process has a modification that is H\"{o}lder continuous for any exponent $\gamma < 1/2$.
\end{thm}


\subsection{Generalization to other summary statistics}\label{sec:corollaries}
The arguments used to prove Theorem \ref{thmfindim} and
\ref{mainthm} apply to other geometric functionals given in terms
of score functions. The key ingedients in the proofs is that the
score function satisfies the inequalities in Lemmas
\ref{varbound_all},  \ref{ximomentbound}, and \ref{lem:fast_xi} .
However, the proofs of these lemmas, especially Lemma
\ref{ximomentbound}, rely on the geometric properties of the
specific score function. Thus, a direct generalization of the
proofs seems only possible for score functions similar in flavour
to the $K$-function. The most obvious such functional is  the
pair correlation function given for isotropic point patterns by
$g(r)=\rho^{(2)}(r)/\rho^2$, where
$\rho^{(2)}(x,y)=\rho^{(2)}(|x-y|)$. The pair correlation
function relates to the $K$-function by $d\kappa_d r^{d-1}g(r)=
K'(r)$. A
kernel estimator for $g(r)$ is given in \cite{chiu} by 
\begin{equation}\label{eq:gndef}
\hat{g}_n(r)=\frac{1}{n d\kappa_d r^{d-1}}\sum_{x\in \PP_n} \sum_{y\in \PP_n} k(r-|x-y|)e_{2,n}(x,y),
\end{equation} 
where $k$ is a compactly supported kernel function which is $C^1$
on its support. Another related functional is the nearest
neighbor function given by $D(r)= \mathbb{P}_o((\PP\backslash
\{o\})\cap B_r(o)\neq \emptyset)$, which can be estimated
\citep{chiu} by 
\begin{equation}\label{eq:Dndef}
\hat{D}_n(r)=\frac{1}{\rho n}\sum_{x\in \PP_n} \mathds{1}_{\{(\PP_n\backslash \{x\})\cap B_r(x)\neq \emptyset\} }e_{4,n}(x,y).
\end{equation}

\begin{cor}\label{cor:pcf}
	Let $\PP$ be a conditionally $m$-dependent point process  with fast decay of correlations satisfying conditions {\bf (M)} and {\bf (R)} or a Gibbs point process of class $\mathbf{\Psi}^*$.  Let $\hat{g}_n$ be as in \eqref{eq:gndef} with $k=f\mathds{1}_{[-\delta,\delta)}$ where $f$ is $C^1$ on $[-\delta,\delta]$, and let $\hat{D}_n$ be as in \eqref{eq:Dndef}.
	The processes
	\begin{align*}
	&\big\{\sqrt{n}(\hat{g}_n(r)-\E\hat{g}_{n}(r)) \big\}_{r\in [r_0+\delta,R-\delta]} \\
	&\big\{\sqrt{n}(\hat{D}_n(r)-\E\hat{D}_{n}(r)) \big\}_{r\in [r_0+\delta,R-\delta]}
	\end{align*}
	converge weakly in Skorokhod topology to a centered Gaussian process, which has a modification that is H\"{o}lder continuous for any exponent $\gamma < 1/2$.
\end{cor}
In the case of $D_n$, the score function
$\xi=\mathds{1}_{(\PP\backslash \{x\})\cap B_r(x)\neq \emptyset}$
can be bounded by the score function \eqref{score} of the
$K$-function, which can be used to generalize Lemmas
\ref{ximomentbound}-\ref{lem:fast_xi}. Only Lemma
\ref{varbound_all} needs a new proof, which can be given in a way
similar to Lemma \ref{c4bound}. 

The case of $g_n$ can also be shown by a straightforward generalization of Lemmas \ref{varbound_all}, \ref{ximomentbound}, and \ref{lem:fast_xi}.
Alternatively, the result could be derived directly from Theorem~\ref{mainthm} by writing
\begin{equation*}
\hat{g}_n(r)= f(-\delta)\hat{K}_{e,n}(r+\delta) - f(\delta)\hat{K}_{e,n}(r-\delta) + \int_{r-\delta}^{r+\delta} \hat{K}_{e,n}(s) f'(r-s)\d s 
\end{equation*}
and applying the continuous mapping theorem, noting that the functional
\begin{align*}
\beta \mapsto  \Big(f(-\delta)\beta(r+\delta) - f(\delta)\beta(r-\delta) + \int_{r-\delta}^{r+\delta} \beta(s) f'(s-r)\d s \Big)_{r\in [r_0+\delta,R-\delta]},
\end{align*}
where $\beta$ is a cadlag function on $[r_0,R]$, is continuous in the Skorokhod topology, see arguments of \cite[Cor. 3.4]{bchs20}. 

In \cite{kclt}, a functional limit theorem was shown in the Poisson case for the more general multi-parameter $K$-function 
	\begin{equation*}
	K(r_1,\ldots,r_d) = \frac{1}{n\rho^{2}}\E \sum_{x\in \PP_n}
	\sum_{y\in \PP \setminus \lbrace x \rbrace }
	\mathds{1}_{\{x-y \in \prod_{i=1}^d [0,r_i] \}} 
	= \frac{1}{\rho}\E_o
	\sum_{y\in \PP\backslash \{o\}}
	\mathds{1}_{\{y \in \prod_{i=1}^d [0,r_i] \}}.
	\end{equation*}
	The results of this paper easily generalize to the
	multiparameter $K$-function, using a multiparameter version
	of the tightness criterion \cite[Lem. 3]{heinrichschmidt}.

A functional central limit theorem for persistence diagrams was
shown in \cite{bchs20} for conditionally $m$-dependent processes.
The proofs easily generalize to Gibbs point processes using Lemma
\ref{lem:fast_xi} of the present paper. The only
case where one has to be careful
is when a hardcore radius is present since this may rule out the
possibility of a variance lower bound in certain parts of the
persistence diagram.


\section{Statistics}\label{sec:gof}

\subsection{Goodness-of-fit tests}
\label{sec:gooftest}

Consider the problem of how well an observed point pattern $\xx$
fits an assumed model $\PP_0$. We
present how the results of
Theorem~\ref{mainthm} can be used for construction of  
goodness-of-fit tests by considering a
Kolmogorov-Smirnov type test based on Ripley's $K$-function. We assume that all parameters are known. 
Let $\mathcal{H}_0$ be the hypothesis that $\xx$ is a realisation
of $\PP_0$. If $\mathcal{H}_0$ is true, then for a given
$R>0$, 
\begin{equation*}
\sup_{r \in [0,R]}	
|\sqrt{n}(\hat{K}_{e,n}(r)-\E\hat{K}_{e,n}(r))|
\approx_{n \rightarrow \infty}
\sup_{r \in [0,R]}	|Y(r)|
\end{equation*}
where $Y(r)$ is a centered Gaussian process with covariance
$C_{\PP_0}$ given by the limiting covariance function
in~\eqref{VarLim}. Let $q_\alpha$ denotes the quantile of
$\sup_{r \in [0,R]}	|Y(r)|$ defined by $P(\sup_{r \in [0,R]}
|Y(r)| \leq q_\alpha) = 1-\alpha$. It follows that under
$\mathcal{H}_0$,
\begin{equation}\label{eq:quantileKS}
\P(\sup_{r \in [0,R]}
|\sqrt{n}(\hat{K}_{e,n}(r)-\E\hat{K}_{e,n}(r))| \leq q_\alpha)
\approx_{n \rightarrow \infty} 1-\alpha.	
\end{equation}
Note that in place of the supremum in~\eqref{eq:quantileKS}, we
could have used the integral over $[0,R]$. Altenatively, one might use either the Cramer-von Mieses type test or the $\chi^2$-test presented
in~\cite{heinrich91}.

\subsection{Simulation study} 

We illustrate the goodness-of-fit test presented above in a
simulation study where we estimate the rejection rate of
$\mathcal{H}_0$ when $\xx$ is indeed a realisation of $\PP_0$ and
when it is a realisation of an alternative
model having the same intensity as $\PP_0$. All point processes
considered have intensity $1$. Each rejection rate has been
estimated, with $\alpha=5\%$, by simulating $10000$ point
patterns $\xx$ on $W_n$ for $n=20^2, 50^2, 100^2, 200^2$.
Moreover, the statistic in~\eqref{eq:quantileKS} has been
computed for $R=1,2,3,4,5$. All the simulations have been done in
R with the package \texttt{spatstat} and the border edge
correction was used.

To shorten, we denote a Poisson point process by $Poi$,  
a log-Gaussian Cox process (LGCP) with exponential covariance
function, variance $\sigma^2$, and scale parameter $a$ by
$LGCP_a(\sigma^2)$, and a Strauss point process with interaction
parameter $\gamma$ and interaction radius $0.4$ by
$Str_{0.4}(\gamma)$. The dependence on the activity parameter
$\beta$ of the Strauss process is omitted in the notation as we
always determine it by simulation so that the intensity of the
process is $1$. 


We first let $\PP_0$ be a Poisson point process with known
intensity $\rho=1$. In that case we have $K(r)=\pi r^2$ and
$C_{\PP_0}(r_1,r_2)=2\pi\min(r_1,r_2)^2/\rho^2+4\pi^2
r_1^2 r_2^2 /\rho$. We have discretized the segment $[0,R]$ by
steps of $0.1$ and estimate empirically the quantile $q_\alpha$
in~\eqref{eq:quantileKS} with $100000$ realisations of a
multivariate centered Gaussian random variable with covariance
matrix $(C_{\PP_0}(r_1,r_2))_{r_1,r_2 \in D_R}$, where $D_R
= \lbrace 0.1x, \ x=0,1,2,\ldots, 10R  \rbrace$. We estimate the
rejection rate of $\mathcal{H}_0$ over $10000$ realisations of
$\PP_0$ and $10000$ realisations of an $LGCP_2(0.2)$ model. The results are
reported in Table~\ref{table:poisson1}.

Second, we repeat the same procedure but with $\PP_0\sim
LGCP_2(\sigma^2)$ for $\sigma=0.2, 1$, and $\PP_0\sim
Str_{0.4}(\gamma)$, for $\gamma=0.2, 0.5, 0.8$. Using 100000
realisations, the intensity of the point process is estimated to
be $1$ when $\beta=1.556, 1.298, 1.107$ for $\gamma=0.2,0.5,0.8$,
respectively. Note that
Inequality~\eqref{eq:tau_condition} holds for these
values. Then, $C_{\PP_0}$ and $\E\hat{K}_{e,n}(r)$ have been
estimated using $10000$ realisations of $LGCP_2(\sigma^2)$ and
$Str_{0.4}(\gamma)$. The rejection rates have been estimated over
$10000$ realisations of $\PP_0$ and $10000$ realisations of a
Poisson process with intensity $1$. Due to computational power
limitation, we have only run perfect
simulations of the Strauss process up to $n=100^2$. We report the
results for  $LGCP_2(0.2)$ and $Str_{0.4}(0.2)$ in
Tables~\ref{table:LGCP02}  and \ref{table:Straussgam02},
respectively, and comment on the remaining cases below.

\begin{table}
	\resizebox{\textwidth}{!}{
		\begin{tabular}{lllllllllll}
			$\PP_0 \sim Poi(1)$			& \multicolumn{5}{c}{$\xx\sim Poi(1)$}						  & \multicolumn{5}{c}{$\xx\sim LGCP_2(0.2)$}   \\\cmidrule(lr){1-1} \cmidrule(lr){2-6}\cmidrule(lr){7-11}
			$W_n$ | $[0,R]$	& $[0,1]$ 	& $[0,2]$ 	& $[0,3]$ 	& $[0,4]$ 	& $[0,5]$ & $[0,1]$ & $[0,2]$ & $[0,3]$ & $[0,4]$ & $[0,5]$ \\  \cmidrule(lr){1-1}\cmidrule(lr){2-6} \cmidrule(lr){7-11}
			$[0,20]^2$  	&  8.0		& 6.8		& 7.0		& 7.7		& 7.8     & 48.5 	& 45.0	  &  43.7	&  43.5	  & 44.4    \\
			$[0,50]^2$  	&  7.1		& 5.2		& 5.6		& 5.6		& 6.0     & 74.3    & 63.1    &  55.8	&  51.1	  & 48.4    \\
			$[0,100]^2$ 	&  7.2		& 6.0		& 5.8		& 5.7		& 5.8     & 97.3    & 89.8    &  78.4	&  67.7	  & 59.4    \\
			$[0,200]^2$ 	&  6.8		& 5.2		& 4.8		& 4.8		& 4.8     & 100     & 100     &  99.4   &  96.0	  & 88.8    \\
		\end{tabular}
	} \caption{Rejection rate (in
	percent) of the null hypothesis  that $\xx$ is a
	realisation of a Poisson point process
	with intensity $1$ on $W_n$, when  $\xx$ in fact comes from either the null model or an $LGCP_2(0.2)$ model, respectively,  using the test
		statistics in~\eqref{eq:quantileKS} on various intervals $[0,R]$.
	}
	\label{table:poisson1}
\end{table}

\begin{table}[h]
	\resizebox{\textwidth}{!}{
		\begin{tabular}{lllllllllll}
			$\PP_0 \sim LGCP_2(0.2)$	& \multicolumn{5}{c}{$\xx \sim LGCP_2(0.2)$}						  & \multicolumn{5}{c}{$\xx\sim Poi(1)$} \\\cmidrule(lr){1-1} \cmidrule(lr){2-6}\cmidrule(lr){7-11}
			$W_n$ | $[0,R]$	& $[0,1]$ 	& $[0,2]$ 	& $[0,3]$ 	& $[0,4]$ 	& $[0,5]$ & $[0,1]$ & $[0,2]$ & $[0,3]$ & $[0,4]$ & $[0,5]$ \\  \cmidrule(lr){1-1}\cmidrule(lr){2-6} \cmidrule(lr){7-11}
			$[0,20]^2$  	& 3.1 		& 3.3		& 3.9		& 4.2		& 4.4     & 0     & 0     & 0 	    & 0 	  & 0     \\
			$[0,50]^2$  	& 3.8 		& 4.0		& 4.4		& 4.6		& 4.7     & 1.0    & 0    & 0 	    & 0 	  & 0     \\
			$[0,100]^2$ 	& 4.1 		& 4.4		& 4.8		& 4.9		& 4.8     & 90.1    & 38.3    & 6.9 	& 0	      & 0    \\
			$[0,200]^2$ 	& 4.3 		& 4.2		& 4.6		& 4.5		& 4.5     & 100    & 100    & 99.6    & 77.9	  & 28.2   \\
		\end{tabular}
	}  \caption{Rejection rate (in
	percent) of the null hypothesis  that $\xx$ is a
	realisation of an  $LGCP_2(0.2)$ on $W_n$, when  $\xx$ in fact comes from either the null model or a $Poi(1)$ model, respectively,  using the test
	statistics in~\eqref{eq:quantileKS} on various intervals $[0,R]$.}
	\label{table:LGCP02}
\end{table}


\begin{table}[h]
	\resizebox{\textwidth}{!}{
		\begin{tabular}{lllllllllll}
			$\PP_0 \sim Str_{0.4}(0.2)$	& \multicolumn{5}{c}{$\xx\sim Str_{0.4}(0.2)$}						  & \multicolumn{5}{c}{$\xx\sim Poi(1)$}     \\ \cmidrule(lr){1-1} \cmidrule(lr){2-6}\cmidrule(lr){7-11}
			$W_n$ | $[0,R]$	& $[0,1]$ 	& $[0,2]$ 	& $[0,3]$ 	& $[0,4]$ 	& $[0,5]$ & $[0,1]$ & $[0,2]$ & $[0,3]$ & $[0,4]$ & $[0,5]$ \\ \cmidrule(lr){1-1}\cmidrule(lr){2-6} \cmidrule(lr){7-11}
			$[0,20]^2$  	& 5.5 		& 6.0		& 6.4		& 7.0		& 7.5     & 34.6   	& 14.6     & 14.3 	& 14.4    & 14.7     \\
			$[0,50]^2$  	& 5.4 		& 5.3		& 5.6		& 5.6		& 5.9     & 100  	& 18.6     & 13.1 	& 11.8    & 11.7     \\
			$[0,100]^2$ 	& 4.9 		& 4.9		& 5.0		& 5.0		& 5.0     & 100  	& 57.0     & 45.0 	& 43.7 	  & 46.6    \\
		\end{tabular}
	} 
	 \caption{Rejection rate (in
		percent) of the null hypothesis  that $\xx$ is a
		realisation of a $Str_{0.4}(0.2)$ on $W_n$, when  $\xx$ in fact comes from either the null model or a $Poi(1)$ model, respectively,  using the test
		statistics in~\eqref{eq:quantileKS} on various intervals $[0,R]$.}
	\label{table:Straussgam02}
\end{table}

\subsection{Discussion}

The results reported on
Tables~\ref{table:poisson1}-\ref{table:Straussgam02} are in
agreement with Theorem~\ref{mainthm}. When $\PP_0$ is a Poisson
point process or $LGCP$ with variance $\sigma^2=0.2$, we recover
the correct type I error rate as soon as $n\geq 100^2$ and $R\geq
2$. When $\PP_0$ is a LGCP with larger variance $\sigma^2=1$, the
type I error was always estimated around $2\%$ and we actually
need $n\geq 300^2$ to recover the correct type I error meaning
that results in this setting would hold only when a very large
number of points are observed. 
In the case of
the Strauss process, the correct type I error is always found
when $n\geq 50^2$.

When tested against an alternative hypothesis,
the test is good at detecting deviations
from the Poisson point process, see Table~\ref{table:poisson1}.
When $\PP_0\sim LGCP_2(0.2)$ or $LGCP_2(1)$, results are bad.
This is due to the large variance of $\hat{K}_{e,n}(r)$ in these
cases leading to a large value of the quantile $q_\alpha$.
Consequently, better results are obtained for small values of $R$
and large windows. In the case of Strauss processes with
$\gamma=0.2$, the assumption that Ripley's $K$-function is the
one of the Strauss process is correctly rejected only for small
values of $R$ and $n\geq 50^2$. As we increase $\gamma$, the
power decreases, as the null model comes closer to the Poisson
process. Hence, depending on the alternative hypothesis
considered, we suggest to use another functional than the
supremum in~\eqref{eq:quantileKS}.

In the case where the number of points is very large,
i.e. more than some millions, it may not be feasible to estimate
$C_{\PP_0}$ and $\E\hat{K}_{e,n}(r)$ through simulation of
realisations of $\PP_0$ as done in our simulation study for the
LGCP and Strauss process. In this case, both $C_{\PP_0}$ and
$\E\hat{K}_{e,n}(r)$ need to be known which,   
although relying on numerical integrations, holds for LGCPs.
Therefore, when the null model is assumed to be a Poisson point
process or LGCP, our goodness-of-fit test only requires the
estimation of the quantile $q_\alpha$ which can easily be
obtained through simulation of Gaussian paths. Therefore,
assuming a known intensity, the goodness-of-fit test proposed in
Section~\ref{sec:gooftest} only take some minutes to return an
output.


\section{Proof of Theorem \ref{thmfindim}}\label{sec:proofs}
\subsection{ The case of conditionally $m$-dependent processes}

In this section, we prove Theorem \ref{thmfindim} for
conditionally $m$-dependent processes.  When no edge corrections
are present, the proof follows  directly from \cite[Thm.
1.13]{yogesh} once we can show the following variance lower
bound.

 
\begin{prop}
	\label{prop:variance lower bound}
	Let $\PP$ be a conditionally $m$-dependent point process having
	fast decay of correlations and satisfying condition {\bf
		(R)}. Then, for all $p\in\N$, $s_1,\ldots,s_p\in \R$, and
	$r_{1},\ldots, r_{p} \in [r_0,R]$, we have
	\begin{equation*}
	\liminf_{n \to \infty} n \Var\Big( \sum_{i=1}^p s_i \hat{K}_n(r_i) \Big) > 0.
	\end{equation*}
\end{prop}

To obtain the proof in the edge corrected case, we need to consider the deviation
\begin{equation}\label{diffK}
\mathcal{E}_{e,n}(r)  = (\hat{K}_n(r) - \E \hat{K}_n(r) ) - (\hat{K}_{e,n}(r) - \E \hat{K}_{e,n}(r) )
\end{equation}
between the edge corrected and uncorrected centered
$K$-functions. We show an upper bound on the variance of
$\mathcal{E}_{e,n}(r)$ in the following lemma together with a
variance bound that we are going to need later in the proof of
Theorem \ref{mainthm}. To state these results, we define
for an interval $I=(r_1,r_2]\subseteq [r_0,R]$,
\begin{equation*}
\hat{K}_{e,n}(I)  =  \hat{K}_{e,n}(r_2)  -  \hat{K}_{e,n}(r_1).	
\end{equation*}

\begin{lem}\label{varbound_all} 
	Let $\PP$ be a point process
	having fast decay of correlations and $r_0,R$ 
	be two constants such that $R>r_0>0$. Let $e=e_{i,n}$, $i\in\{1,\ldots,5\}$, be one
	of the edge correction factors in Section \ref{sec:K}. There
	is a constant $C>0$ such that for all $r\in [r_0,R]$ and
	$I=(r_1,r_2]\subseteq [r_0,R]$,
	\begin{align}\label{varKhat}
	n\Var(\hat{K}_{n,e}(I)) {}&\leq C|I|\\
	n\Var(\mathcal{E}_{n,e}(r)) {}&\leq Cn^{-1/d}.\label{varKtilde}
	\end{align}
\end{lem}

We first show how Theorem \ref{thmfindim} follows from these two
results and then give {their} proofs.

\begin{proof}[Proof of Theorem \ref{thmfindim}($m$-dependent case)]
	We first consider the uncorrected case and show the joint 
	convergence of $(\hat{K}_n(r_1),\ldots,\hat{K}_n(r_p))$.	
	By the Cram\'{e}r-Wold device, it
	suffices to show a 
	pointwise central limit theorem for all
	linear combinations of
	$\hat{K}_n(r_1),\ldots,\hat{K}_n(r_p)$. This follows directly
	from \cite[Thm. 1.13]{yogesh}, where the assumptions are
	satisfied by Lemma \ref{A2} and the asymptotic variance lower
	bound in Proposition \ref{prop:variance lower bound}.
	
	Second, by Lemma \ref{varbound_all}, we have that
	\begin{equation*}
	\lim_{n\to \infty} n\Var(\mathcal{E}_{e,n}(r)) = 0.
	\end{equation*}
	Therefore, for each $j=1,\ldots,p$ we have that 
	\begin{equation*}
	\sqrt{n}(\hat{K}_n(r_j) - \E \hat{K}_n(r_j) ) - \sqrt{n}(\hat{K}_{e,n}(r_j) - \E \hat{K}_{e,n}(r_j))
	\end{equation*}
	converges to 0 in probability, so the statement of Theorem \ref{thmfindim} in the edge corrected case follows from the uncorrected case.
\end{proof}

It remains to prove Proposition \ref{prop:variance lower bound} and Lemma \ref{varbound_all}.

\begin{proof}[Proof of Proposition~\ref{prop:variance lower bound}]
Let  $p\in\N$, $s_1,\ldots,s_p\in \R$ and
$r_0\leq r_{1}<\dotsm<r_{p}\leq R$ be given. For any Borel set  $A \subset \R^{d}$, we introduce the notation
\begin{equation*}
T_{n}(A)= \sum_{i=1}^{p} s_{i} \sum_{x\in A\cap \PP_n}
\sum_{y\in \PP_n\setminus \lbrace x \rbrace} \mathds{1}_{\{|x-y|\leq r_{i}\}}.
\end{equation*} 
As a special case, 
\begin{equation*}
T_n(W_{n}) =
\sum_{i=1}^{p}  s_{i} \sum_{x\in \PP_n}
\sum_{y\in \PP_n \setminus \lbrace x \rbrace} 
\mathds{1}_{\{|x-y|\leq r_{i}\}} = n\rho^{2} \sum_{i=1}^p s_i\hat{K}_n(r_i),
\end{equation*}
so what we need to show is
\begin{equation}\label{eq: var order}
\liminf_{n\to \infty} \frac{1}{n}\Var ( T_n(W_n)) >0.
\end{equation}
	
For any Borel set $A \subset \R^{d}$, we let $\Lambda^{\PP}_{A}=\sigma(\Lambda,\PP\cap A)$ to shorten notation  and recall that $\tilde{R} = \max(m,R)$. For $t>0$, let
 \begin{equation*}
C_{t} = \bigcup_{z \in \Z^{d}} (6 \tilde{R} z + W_{t^{d}}) \cap W_{n} 
\end{equation*}
 be the union of cubes of side length $t$ centered at the vertices of the lattice
$6\tilde{R} \Z^{d}$. Finally, let
\begin{equation*}
  A_{t} = W_{n} \setminus C_{t}.
\end{equation*}
From the law of total variance, it follows that
\begin{equation}\label{eq:bound var 1}
  \Var T_{n}(W_{n}) \geq  \E \Var( T_{n}(W_{n}) | \Lambda^{\PP}_{A_{\tilde{R}}}).
\end{equation}
We have that $T_{n}(W_{n}) = T_{n}(A_{3\tilde{R}}) + T_{n}(C_{3\tilde{R}})$ and $T_{n}(A_{3\tilde{R}})$  depends only on $\PP$ inside
$A_{3\tilde{R}} \oplus B_R(o) \subset A_{\tilde{R}}$, where $\oplus$ denotes
Minkowski set addition. Thus $T_{n}(A_{3\tilde{R}})$  is measurable with respect to $A_{\tilde{R}}$ and hence
$\Lambda^{\PP}_{A_{\tilde{R}}}$. 
Thus, by Lemma \ref{lemma:ineq conditional variance} we have 
\begin{equation}\label{eq:bound var 2}
 \Var( T_{n}(W_{n}) | \Lambda^{\PP}_{A_{\tilde{R}}}) = \Var( T_{n}(C_{3\tilde{R}}) | \Lambda^{\PP}_{A_{\tilde{R}}}).
\end{equation}
The squares in $C_{3\tilde{R}}$ are separated by at least distance $3\tilde{R}$, and
$ T_{n}(C_{3\tilde{R}})$ only depends on the points of $\PP$ within distance $R\leq \tilde{R}$ from $C_{3\tilde{R}}$. Thus, the conditional
$m$-dependence and \eqref{eq:bound var 1}--\eqref{eq:bound var 2}
yield
\begin{equation*}
  \Var T_{n}(W_{n}) \geq \sum_{z \in \Z^{d}}  \E  \Var( T_{n}( 6 \tilde{R} z + W_{(3\tilde{R})^{d}}) | \Lambda^{\PP}_{A_{\tilde{R}}}).
\end{equation*}
Since for all $z\in \Z^{d}$,
$ \PP \cap A_{\tilde{R}} \subset  \PP \cap (6 \tilde{R} z + W_{\tilde{R}^{d}})^c$,
  Lemma~\ref{lemma:ineq conditional variance} yields,
\begin{align}
  \nonumber
  \Var T_{n}(W_{n}){}& \geq
  \sum_{z \in \Z^{d}}  \E  \Var\Big( T_{n}(6 \tilde{R} z + W_{(3\tilde{R})^{d}}) | \Lambda^\PP_{(6\tilde{R}z+ W_{\tilde{R}^{d}})^c} \Big)\\ \label{eq:bound var 3}
{}&\geq
\sum_{\substack{z \in \Z^{d}\\ 6\tilde{R}z+ W_{(5\tilde{R})^{d}} \subset W_n } }  \E  \Var\Big( T_{n}( W_{(3\tilde{R})^{d}}) | \Lambda^\PP_{(W_{\tilde{R}^{d}})^c} \Big).
\end{align}
The last inequality used stationarity and the fact that $T_{n}(6 \tilde{R} z + W_{(3\tilde{R})^{d}})$ only depends on $\PP \cap (6 \tilde{R} z + W_{(5\tilde{R})^{d}})$.

We may assume $s_p\neq 0$. Let
\begin{align}
\begin{split}\label{F1F2}
F_{1} ={}& \lbrace  \forall (x,y) \in (\PP\cap W_{(5\tilde{R})^{d}})_2^{\neq}: |x-y|>R \rbrace \\
F_{2} ={}& \left\lbrace \forall (x,y) \in (\PP\cap W_{(5\tilde{R})^{d}})_2^{\neq}: |x-y|>r_{p-1}\right\rbrace \\
&\cap \left\lbrace  \exists (x,y) \in (\PP\cap W_{(3\tilde{R})^{d}})_2^{\neq}: |x-y|\leq r_{p} \right\rbrace. 
\end{split} 
\end{align}
Then, 
\begin{align*}
F_{1} {}&\subset \lbrace  T_{n}(W_{(3\tilde{R})^{d}}) \in I_{1} \rbrace\\
F_{2} {}&\subset \lbrace | T_{n}(W_{(3\tilde{R})^{d}}) | \in I_{2} \rbrace,
\end{align*}
where $I_{1} =\lbrace 0 \rbrace$ and
$I_{2}= [|s_p|, \infty)$.  Therefore, by applying
Lemma~\ref{23Lem} with $Y= T_{n}(W_{(3\tilde{R})^{d}})$, we have
\begin{align*}
\Var\Big( T_{n}(W_{(3\tilde{R})^{d}} ) | \Lambda^\PP_{(W_{\tilde{R}^{d}})^c} \Big)  &\geq
\frac{s_p^2}{4} \min_{i \in \{1,2\}} \P\Big( |T_{n}(W_{(3\tilde{R})^{d}})| \in I_{i}| \Lambda^\PP_{(W_{\tilde{R}^{d}})^c}\Big) \\
&\geq   \frac{s_p^2}{4} \min_{i \in \{1,2\}} \P\Big(F_{i}| \Lambda^\PP_{(W_{\tilde{R}^{d}})^c} \Big).
\end{align*}
Thus, by Condition {\bf (R)},
\begin{equation*}
\E \Big(\Var\Big( T_{n}(W_{(3\tilde{R})^{d}}) | 
\Lambda^\PP_{(W_{\tilde{R}^{d}})^c} \Big) \Big)> 0.
\end{equation*}
This, together with \eqref{eq:bound var 3}, shows \eqref{eq: var
order} since the number of terms in~\eqref{eq:bound var 3} is of
order $n$.  
\end{proof}

\begin{proof}[Proof of Lemma~\ref{varbound_all}]
	For 	$x\in \R^{d}$, let $A_{I}(x) = \{y\in \R^{d}\mid r_1 < |x-y|\leq r_2\}$ denote the annulus
centered at $x$ and having inner and outer radius $r_{1}$ and $r_{2}$, respectively. Note that  $|A_I(x)| = \kappa_d (r_2^d-r_1^d)$ and that $y\in A_I(x)$ is equivalent to $x\in A_I(y)$.

	Both $n\hat{K}_{e,n}(I)$ and $n\mathcal{E}_{e,n}(r)$ take the form
	\begin{equation}\label{generalK}
	\frac{1}{\rho^2} \sum_{x,y \in (\PP_n)_2^{\neq} } \mathds{1}_{\{x\in A_I(y) \}}e(x,y),
	\end{equation}
	where  $n\hat{K}_{e,n}(I)$ corresponds to $e=e_{i,n}$, and $n\mathcal{E}_{e,n}(r)$ corresponds to $e=1-e_{i,n}$ and $I=(0,r] $.

	Since
	  \begin{align*}
	&\E \bigg(  \sum_{x,y \in (\PP_n)_2^{\neq} } \mathds{1}_{\{r_{1}<|x-y|\leq r_{2}\}}  e(x,y) \bigg)^{2}
	\\&\quad =\E \bigg( \sum_{x,y,u,v \in (\PP_n)_4^{\neq}}   
	{e}(x,y){e}(u,v)\mathds{1}_{\{x\in A_I(y)\}} \mathds{1}_{\{u\in A_I(v) \}}  
	\bigg) \\
	& \quad \quad +
	\E\bigg( \sum_{x,y,u \in (\PP_n)_3^{\neq} } ({e}(x,y) + {e}(y,x))({e}(y,u) + {e}(u,y))
	\mathds{1}_{\{x\in A_I(y)\}}  \mathds{1}_{\{u\in A_I(y) \}}  
	\bigg)\\
	& \quad \quad +
	\E \bigg( \sum_{x,y \in (\PP_n)_2^{\neq} } ({e}(x,y)^2 + {e}(x,y){e}(y,x)) \mathds{1}_{\{x\in A_I(y)\}}  
	\bigg),
	\end{align*}
	we may write the variance of \eqref{generalK} as $\mathcal{I}_{1} + \mathcal{I}_{2} + \mathcal{I}_{3}$ where
	\begin{align*}
	\mathcal{I}_{1}=&\int_{W_{n}^{4}} 
	{e}(x,y){e}(u,v)\mathds{1}_{\{x\in A_I(y)\}}\mathds{1}{\{u\in A_I(v) \}}  \\
	&\quad \times (\rho^{(4)}(x,y,u,v) -\rho^{(2)}(x,y) \rho^{(2)}(u,v))\d x \d y \d u \d v \\
	\mathcal{I}_{2}=&
	\int_{W_{n}^{3}} ({e}(x,y) + {e}(y,x))({e}(y,u) + {e}(u,y))
	\mathds{1}_{\{x\in A_I(y)\}}  \mathds{1}_{\{u\in A_I(y) \}}  
	\rho^{(3)}(x,y,u)  \d x \d y \d u \\
	\mathcal{I}_{3}=&
	\int_{W_{n}^{2}} ({e}(x,y)^2 + {e}(x,y){e}(y,x)) \mathds{1}_{\{x\in A_I(y)\}}  
	\rho^{(2)}(x,y) \d x \d y.
	\end{align*}
		
	First consider $n\hat{K}_{e,n}(I)$.  
	Definition~\ref{def:expdecay} implies that $\rho^{(p)}$ is
	bounded on $\R^{pd}$, see also~\cite[(1.11)]{yogesh}. Moreover, the edge correction factors $e_{i,n} $ are  bounded whenever $n>(2R)^d$. Thus, there exists a
	constant $C_1>0$ such that 
	\begin{equation*}
	\mathcal{I}_{2} 
	\leq
	C_1  \int_{W_{n}} \int_{\R^{d}} \int_{\R^{d}}  \mathds{1}_{ \{x \in A_{I}(y)\} } \mathds{1}_{\{u \in A_{I}(y)\}}
	\d u \d x \d y = C_1 |W_n| |A_I(o)|^2.
	\end{equation*}
	Since $r_{1},r_{2} \in [0,R]$, there exists a constant $C_2$,
	depending only on $R$ and $d$, such that
	$|A_I(o)|=\kappa_d (r_{2}^{d}-r_{1}^{d}) \leq C_2 |I|$. Hence, $\mathcal{I}_{2} \leq C_3 |W_{n}| |I|^{2}$.
	Similarly, there exist constants $C_4,C_5>0$ such that
	\begin{equation*}
	\mathcal{I}_{3}
	\leq 
	C_4 \int_{W_{n}} \int_{\R^{d}} \mathds{1}_{\{x \in A_{I}(y)\}}  \d x \d y 
	\leq {C_5  |W_{n}|} |I|.
	\end{equation*}

	By Definition~\ref{def:expdecay}, there exist constants $C_6,C_7,C_8>0$ and a function $\phi(t) \leq C_6( t^{-(d+1)}\wedge 1)$  such that
	\begin{align*}
	\mathcal{I}_{1}{}&
	\leq C_7 \int_{W_{n}^{4}}
	\mathds{1}_{\{ x\in A_I(y) \}}  \mathds{1}_{\{ u \in A_I(v) \}}
	\phi(\dist(\{x,y\}, \{u,v\})) \d x \d y \d u \d v\\
	&\leq
	C_8 \int_{W_{n}^{4}}
	\mathds{1}_{\{x\in A_I(y) \}}  \mathds{1}_{\{ u \in A_I(v) \}}
	\Big(\frac{1}{\dist(\{x,y\}, \{u,v\})^{d+1}} \wedge 1 \Big) \d x \d y \d u \d v.
	\end{align*}
	Renaming variables appropriately, we may assume that $\dist(\{x,y\}, \{u,v\})=|y-u|$. Then, 
	\begin{align*}
	\mathcal{I}_1 
	{}&\leq C_9 \int_{W_{n}} \int_{\R^{d}}  \int_{\R^{d}} \int_{\R^{d}}
	\mathds{1}_{\{x\in A_{I}(y)\}}  \mathds{1}_{\{u \in A_{I}(v)\} }
	\Big(\frac{1}{|y-u|^{d+1}} \wedge 1\Big)
	\d x \d v \d y \d u \\
	&\leq C_{10} |W_{n}||I|^2.
	\end{align*}
	Together, the bounds on $ \mathcal{I}_1$, $\mathcal{I}_2$, and 
	$\mathcal{I}_3$ show \eqref{varKhat}.
	
	Next consider $\Var(n\mathcal{E}_{e,n}(r))$.	
	For $e=1-{e}_{i,n}(x,y)$, $i=4,5$, we note that $e$ is again bounded. Moreover, $e(x,y)$ vanishes outside $(W_{n,2r})^2$, where we use the notation
	\begin{equation*}
	W_{n,s} = W_n \setminus (W_n \ominus B_s(o)) =W_n\backslash \Big[-\tfrac{n^{1/d }}{2} + s, \tfrac{n^{1/d}}{2} - s\Big]
	\end{equation*}
	for the set of points that are within distance $s$ from the boundary of $W_n$. This implies that the factor $n=|W_n|$ in the variance bound can be replaced by $|W_{n,2r}|$.
	Note that 
	\begin{equation*}
	|W_{n,2r}|   = n - (n^{1/d} - 4r)^d  \leq C_{11}n^{1-1/d}.
	\end{equation*}
	Thus, 
	\begin{equation}\label{varKebound}
	n\Var(\mathcal{E}_{e,n}(r)) \leq C_{12}n^{-1/d}.
	\end{equation}  
	
	For $e=1-e_{2,n}$, we note that whenever $|x-y|\leq r$, we have $W_n \ominus B_{r}(o) \subseteq W_n \cap (W_n + x - y)$, so for $n$ large enough
	\begin{equation}\label{volbound}
	|W_n \cap  (W_n + x-y)| \geq (n^{1/d} - 2r)^d \geq  n - C_{13}n^{1-1/d}.
	\end{equation}
	Thus, for some $C_{14}>0$,
	\begin{equation*}
	|{e}_{2,n}(x,y) | = \frac{|W_n| - |W_n\cap (W_n + x-y) |}{|W_n \cap (W_n + x-y)|}\leq  \frac{C_{13}n^{1-1/d}}{n-C_{13}n^{1-1/d}} \leq C_{14}n^{-1/d}. 
	\end{equation*} 
	Using this and proceeding as in the case of $\hat{K}_{e,n}(I)$ yields \eqref{varKtilde}.
	
	The case  $e=1-e_{3,n}$ is very similar. For any rotation $\eta \in SO(d)$, apply the inequality \eqref{volbound} to $|W_n + \eta (x- y)|$. Averaging over all rotations, the inequality is preserved. Proceeding as before, we obtain \eqref{varKtilde}. 
\end{proof}

\subsection{Proof of Theorem \ref{thmfindim} in the Gibbs case}

\begin{proof}[Proof of Theorem \ref{thmfindim}(Gibbs case)]
	We first consider the case without edge corrections.  The modification for edge corrections follows from Lemma \ref{varbound_all} as in the case of conditionally $m$-dependent point processes.
	
	We again use the Cram\'{e}r-Wold device, that is, we need to show a central limit theorem for all linear combinations of $\hat{K}_n(r_1),\ldots,\hat{K}_n(r_p)$, $p\geq 1$,  $0\leq r_1\leq \dotsm \leq r_p \leq R$. Such a linear combination  corresponds to a score function of the type
	\begin{equation*}	
	\xi = \sum_{i=1}^p a_i \xi_{r_i}.
	\end{equation*}
	By \cite[Thm 1.2]{gibbsCLT}, it is enough to show the following four points:
	\begin{enumerate}
		\item $\xi$ is exponentially stabilizing in the sense of \cite[(1.8)]{gibbsCLT}.
		
		\item $\xi$ is translation invariant. 
		
		\item $\xi$ verifies
		\begin{equation*}
		\sup_n \sup_{x\in W_n} \E[\xi(x,\PP\cup{x})^4] < \infty. 
		\end{equation*}

		\item There exists $s>0$ such that
		\begin{equation*}
		\inf_{t\geq s} \E \Var (\sum_{x \in \PP\cap W_t } \xi(x,\PP) | \PP\cap W_s^c) \geq b_0.
		\end{equation*}
	\end{enumerate}
	
	Point 1. is trivially satisfied, since $\xi$ has finite radius of stabilization at most $R$, i.e., $\xi(x,\mathcal{X})$ depends only on $ \mathcal{X}\cap B_R(x)$,  and 2. holds by definition. Point 3. holds because
	\begin{equation*}
	\E(\xi(x,\PP\cup{x})^4) \leq \sum_{i=1}^p |a_i|   \E[\PP(B_R(x))^4] = \sum_{i=1}^p |a_i|  \E[\PP(B_R(o))^4]<\infty.
	\end{equation*}
The last inequality holds because $\PP$ can be given as a thinning of a Poisson process of intensity $\tau$, which has finite moments.
	
	To show 4., we take $s=R^d$ and note that by 
	Lemma \ref{lemma:ineq conditional variance}, and since $\xi$
	has stabilization radius $R$,
		\begin{equation*}
		\inf_{t\geq R^d} \E \Var \Big(\sum_{x \in \PP\cap W_t } \xi(x,\PP) | \PP\cap W_{R^d}^c\Big)=\inf_{t\in [R^d,(2R)^d]}  \E \Var \Big(\sum_{x \in \PP\cap W_t } \xi(x,\PP) | \PP\cap W_{R^d}^c\Big).
		\end{equation*}
		Place two small open balls, $A_1$ and $A_2$, inside $W_{R^d}$ at least distance $r_{p-1}$ apart and both contained inside a ball of diameter $r_p>r_{p-1}$. 
		Define the events
		\begin{align*}
		E'{}&=\{\PP(W_{(3R)^d}\setminus W_{R^d})=0\}\\
		E_1{}&=\{\PP(W_{R^d})=0 \}\\
		E_2{}&=\{\PP(A_1) = \PP(A_2) =1, \PP(W_{R^d}\setminus (A_1\cup A_2)) =0 \}.
		\end{align*}
		Then for every $t\in [R^d,(2R)^d]$, 
		the event $E'\cap E_i$ is contained in the event 		
		\begin{equation*}
					 \bigg\lbrace\big|
			 \sum_{x \in \PP\cap W_{t} } \xi(x,\PP) \big|\in I_i
			 \bigg\rbrace,
		\end{equation*}
		where $I_1=\{0\}$ and $I_2= [|a_p|,\infty)$, and hence by Lemma \ref{23Lem},
		\begin{equation*}
		\Var \bigg(\sum_{x \in \PP\cap W_{t} } \xi(x,\PP) | \PP\cap W_{R^d}^c\bigg) \geq  \frac{a_p^2}{4}\min_{i=1,2} \mathbb{P}(E'\cap E_i|  \PP\cap W_{R^d}^c ).
		\end{equation*}
		Since $E'$ is measurable with respect to $\PP\cap W_{R^d}^c$, we find
		\begin{align*}
		&\mathbb{P}(E'\cap E_2 |  \PP\cap W_{R^d}^c ) \\
		&= \mathds{1}_{E'} \mathbb{P}(\PP(A_1) = \PP(A_2) = 1,\PP(W_{R^d}\setminus(A_1\cup A_2))=0| \PP\cap W_{R^d}^c).
		\end{align*}
		We now apply \cite[Lem. 2.2]{gibbsCLT} to show that
		\begin{align*}
		&\mathbb{P}(E'\cap E_2 |  \PP\cap W_{R^d}^c )\\
		& = \mathds{1}_{E'}\mathbb{P}(\PP(A_1) = \PP(A_2) = 0,\PP(W_{R^d}\setminus(A_1\cup A_2))=0| \PP\cap W_{R^d}^c)\\&
		\geq  e^{-\tau R^d} \mathds{1}_{E'}\mathbb{P}(\PP(A_1) = \PP(A_2) = 1,\PP(W_{R^d}\setminus(A_1\cup A_2))=0| \PP\cap W_{R^d}^c).
		\end{align*}
		Hence 
		\begin{align}\label{eq:gibbs_lower_bound}
		&\E \bigg( \Var \bigg(\sum_{x \in \PP\cap W_{t} } \xi(x,\PP) | \PP\cap W_{R^d}^c\bigg)\bigg) \\ \nonumber
		&\geq \frac{a_p^2}{4} \E (\min_{i=1,2} \mathbb{P}(E'\cap E_i|  \PP\cap W_{R^d}^c ) )\\
		 &\geq  \frac{a_p^2}{4}  e^{-\tau R^d} \E ( \mathds{1}_{E'} \mathbb{P}(\PP(A_1) = \PP(A_2) = 1,\PP(W_{R^d} \setminus(A_1\cup A_2))=0 | \PP\cap W_{R^d}^c)). \nonumber
		\end{align}
		We need to check that the latter expectation is strictly
		positive. Recall that the Gibbs point process has the
		density \eqref{eq:gibbs_def} with respect to the Poisson
		process $Q$ on $W_{R^d}$ when conditioning on $\PP\cap
		W_{R^d}^c=\mathcal{X}_0$. This density is bounded  from
		below on $E_2$ by some $c>0$
		uniformly in all values of
		$\mathcal{X}_0 \in E'$. This is because the denominator
		in the density is bounded from above by 1, and the
		numerator is bounded from below on $E_2\cap E'$. Indeed,
		the maximal number of points in $W_{(3R)^d}$ is two, and
		these points are at distance of at least $r_0$ from each
		other and at least $R$ from any other points. Hence
		$\Delta^\Psi(\mathcal{X},\mathcal{X}_0)$  is bounded from
		above, resulting in a lower bound on the numerator. With
		this bound on the density, we have on $E'$
		\begin{align*}
		&\mathbb{P}(\PP(A_1) = \mathds{1}_{E'} \PP(A_2) = 1,\PP(W_{R^d} \setminus(A_1\cup A_2))=0 | \PP\cap W_{R^d}^c)\\
		&\geq c\mathbb{P}(\mathcal{Q}(A_1) = \mathcal{Q}(A_2) = 1,\mathcal{Q}(W_{R^d} \setminus(A_1\cup A_2))=0 ) >0.
		\end{align*}
		Positivity of \eqref{eq:gibbs_lower_bound} now follows
		because $\mathbb{P}(\PP \in E')\geq
		\mathbb{P}(\mathcal{Q}\in E')$ since $\PP$ can be given
		as a thinning of a Poisson process $\mathcal{Q}$, see
		Appendix \ref{app:thinning}.		 
\end{proof}


\section{Proof of Theorem \ref{mainthm}}\label{s:tightsec}

The proof of Theorem \ref{mainthm} is based on Lemma~\ref{c4bound} below, which provides
a bound on the fourth cumulant $c^4$. To state 
Lemma~\ref{c4bound}, let $r\in [r_0,R] $ and
 $I=(r_1,r_2] \subseteq [r_0,R]$ and introduce the notations
\begin{align*}
\bar{K}_{e,n}(r) {}&= \hat{K}_{e,n}(r)-\E \hat{K}_{e,n}(r)\\
\bar{K}_{e,n}(I) {}&= \bar{K}_{e,n}(r_2) -\bar{K}_{e,n}(r_1).
\end{align*}
Two intervals are called neighboring if they share exactly one
end point. 

\begin{lem}\label{c4bound} 
	Let $\PP$ be a point process having
	fast decay of correlations and satisfying Condition {\bf (M)}
	or a Gibbs point process of class $\mathbf{\Psi}^*$. Let
	$e=e_{i,n}$, $i\in \{1,\ldots,5\}$, be one of the edge
	corrections listed in Section \ref{sec:K} and let
	$I_{1},I_{2}\subseteq [r_0,R]$ be two neighboring intervals.
	Then there is a constant $C>0$ such that
	\begin{equation*}
	c^4(n\bar{K}_{e,n}(I_1),n\bar{K}_{e,n}(I_1),n\bar{K}_{e,n}(I_2),n\bar{K}_{e,n}(I_2)) \leq C n |I_1|^{3/4} |I_2|^{3/4}.
	\end{equation*}
\end{lem}
The proof is given in Section \ref{sec:c4 lem proof}.  We first
show how Lemma~\ref{c4bound} implies
Theorem~\ref{mainthm}.

\begin{proof}[Proof of Theorem \ref{mainthm}]
	Let $k_r$ denote the limit of $\sqrt{n}\bar{K}_{n,e}(r)$ and
	let $I=(r_1,r_2]$. Since by Theorem~\ref{thmfindim},
	$\sqrt{n}\bar{K}_{e,n}(I)$ converges in distribution to
	$(k_{r_2}-k_{r_1})$ as $n$ tends to infinity,  
	the Portmanteau theorem and Lemma
	\ref{varbound_all} yield
	\begin{equation}\label{eq:limProb}
	\Var(k_{r_2}-k_{r_1}) =\E [(k_{r_2}-k_{r_1})^2] \leq n\Var \bar{K}_{e,n}(I) \leq C|I|.
	\end{equation}
	
	According to \cite[Lem. 3]{heinrichschmidt}, convergence in
	Skorokhod topology is ensured if we can show three
	properties. The first is convergence
	of finite-dimensional distributions, which follows from
	Theorem \ref{thmfindim}. The second
	property we need to check is that
	\begin{equation*}
	\lim_{r\to R} P( |k_R-k_r| \geq \eps ) =0
	\end{equation*}
	for any $\eps> 0$. This follows from \eqref{eq:limProb} and
	the Chebyshev inequality. Finally, we need to show that there
	is an $\eps >0 $ and a constant $C>0$ such that for all
	neighboring intervals $I_1$ and $I_2$,
	\begin{equation*}
	\E(n^2\bar{K}_{e,n}(I_1)^2\bar{K}_{e,n}(I_2)^2) \leq C|I_1|^{1/2 + \eps}|I_2|^{1/2 + \eps}.
	\end{equation*}
	This follows by applying Lemma \ref{varbound_all} and
	Lemma~\ref{c4bound} to the following formula  with $X=\sqrt{n}\bar{K}_{e,n}(I_1)$
	and $Y=\sqrt{n}\bar{K}_{e,n}(I_2)$: 
	\begin{align*}
	\E(X^2Y^2){}& = c^4(X,X,Y,Y) + \Var(X)\Var(Y) + 2\Cov(X,Y)^2\\ 
	&\leq c^4(X,X,Y,Y) + 3\Var(X)\Var(Y).
	\end{align*}
	
	H\"{o}lder continuity of  $(k_r)_{r\in [r_0,R]}$ follows from
	the Kolmogorov continuity theorem \citep{kallenberg}. Indeed,
	we know from Theorem \ref{thmfindim} that for any  $r_0\leq
	r_1 < r_2 \leq R$, $(k_{r_2}-k_{r_1})$ is gaussian with mean
	0. It follows from \eqref{eq:limProb} and Jensen's
	inequality that for any integer $m\geq 1$, 
	\begin{equation*}
	\E [(k_{r_2}-k_{r_1})^{2m} ]\leq C_k|r_2-r_1|^{m}.
	\end{equation*}
	Hence $(k_r)_{r\in [r_0,R]}$ has a H\"{o}lder continuous modification for any exponent $\gamma < \frac{1}{2}-\frac{1}{2m}$.
\end{proof}

It remains to show Lemma \ref{c4bound}. 
The proof is based on
decompositions of cumulant measures into mixed-moments and semi-clusters. The necessary background is given in Section \ref{s:prel-moment-decomp}. Two preliminary lemmas are shown in Section \ref{sec:xi_weighted_bound} and  \ref{sec:xi_weighted}, respectively, before the proof is presented in Section \ref{sec:c4 lem proof}. 


\subsection{Background on decomposition of cumulant measures}
\label{s:prel-moment-decomp}

 In this section, we recall the 
necessary background on decomposition of cumulant measures
into mixed-moments and semi-clusters. We refer the reader to~\cite{raic3} for a detailed presentation.

\subsubsection*{Moment and cumulant measures}
Let $\PP$ be a simple point process on $\R^{d}$. We
define a marked version $\breve{\PP} = \PP \times \{1, 2\}$ on
$\bR^d = \R^d \times \{1, 2\}$ and let $\bPPn = \PPn \times
\{1,2\}$.   For any extended score function
$\breve{\xi}:\bR^d\times \mathcal{N} \to \R^d$, we define the
random measure on $\bR^d$
\begin{equation}\label{mun}
\mu_{n} = \sum_{\bx \in \bPPn} \breve{\xi}(\bx, \PPn) \delta_{\bx}.
\end{equation}
Then, following~\cite[Section 3.1]{raic3}, the \emph{$k$-th moment measure}
$M^{k}(\mu_{n})=M_n^k$ is defined as 
\begin{align*}
\int_{(\bR^d)^k} \mathbf{f}(\bxx) M^{k}_n (\mathrm{d} \bxx)
{}&=\E\Big[
\Big(\int_{\bR^d } f_1\d\mu_n  \Big)
\dotsm
\Big(\int_{\bR^d } f_k\d\mu_n  \Big)
\Big]\\
&=
\E\Big[
\Big(\sum_{\bx \in \bPPn} \breve{\xi}(\bx, \PPn) f_{1}(\bx)\Big)
\dotsm
\Big(\sum_{\bx \in \bPPn} \breve{\xi}(\bx, \PPn) f_{k}(\bx)  \Big)
\Big],
\end{align*}
for any non-negative measurable function $\mathbf{f}=f_{1}
\otimes \ldots \otimes f_{k}$, where each $f_{i}$ is defined on
$\bR^{d}$.  Similarly, the
\emph{$k$-th cumulant measure} $c^k(\mu_n)=c_n^k$ is given by
\begin{align} \label{cummesdef}
\int_{(\bR^d)^k} \mathbf{f}(\bxx) c^{k}_n (\mathrm{d} \bxx)
{}&=
c^k\Big(
\int_{\bR^d } f_1\d\mu_n ,
\ldots,
\int_{\bR^d } f_k\d\mu_n 
\Big)\\ \nonumber
&=
c^k\Big(
\sum_{\bx \in \bPPn} \breve{\xi}(\bx, \PPn) f_{1}(\bx),
\ldots,
\sum_{\bx \in \bPPn} \breve{\xi}(\bx, \PPn) f_{k}(\bx)  
\Big).
\end{align}
We have the following expression for $c_n^k$ in terms of moment measures
\begin{equation}\label{eq:cumulant xi n}
c_{n}^{k}
=
\sum_{\{T_1, \dots, T_p\}
	\preceq \{1,\ldots,k\}}(-1)^{p - 1} (p - 1)!\, M^{T_1}_n \cdots M^{T_p}_n, 
\end{equation}
where $\{T_1, \dots, T_p\} \preceq \{1, \dots, k\}$ denotes the set of all partitions of $ \{1, \dots, k\}$ into $p$ non-empty sets $T_1,\ldots,T_p$, $p=1,\ldots,k$, and $M_n^T$ is defined as $M_n^{|T|}$ on the coordinates in $T$.

\subsubsection*{Decomposition of moment measures}

Let $T\subseteq \{1,\ldots,k\}$. For $\bxx \in (\bR^{d})^k$, we let $\bxx_{T}$ be the
projection of $\bxx$ onto the coordinates in $T$, and we denote by $\xx$ and $\xx_T$ the unmarked points corresponding to $\bxx$ and $\bxx_T$, respectively. 
 We define the diagonal of $(\bR^{d})^k$ to be 
  $\Delta = \{ \bxx \in (\bR^{d})^k : \xx = (x,\ldots,x) \}$, and $\Delta^T\subseteq (\bR^{d})^T $ is defined similarly. More generally, for a partition $T_1,\ldots,T_p \preceq \{1, \dots, k\} $, we define $\Delta^{T_1,\ldots,T_p} = \Delta^{T_1}\times \dotsm \times \Delta^{T_p}$. Finally, we let $\pi: \Delta^{T_1,\ldots,T_p} \to \R^p$ be the map that ignores marks and repeated points.

For a given partition $T_{1}, \ldots, T_{p}\preceq \{1,\ldots,k\}$, 
the \emph{mixed $\xi$-moments} of $\mu_{n} $ are defined for $\bxx \in \Delta^{T_1,\ldots,T_p}$ by
\begin{equation}\label{eq:xi-mixed}
m^{T_1, \dots, T_p}_n(\bxx)
=
\E_{\pi(\bxx)}\big[\breve{\xi}(\bx_{1}, \PP_n) \cdots \breve{\xi}(\bx_{k}, \PP_n)\big] \rho^{(p)}(\pi(\bxx)).
\end{equation}
Then we obtain the following expression for $M^{k}_n$,
\begin{equation}\label{eq:density mixed moment}
\d M^{k}_n
=
\sum_{T_1, \ldots, T_{p} \preceq \{1,\ldots,k\} }
m_n^{T_1, \dots, T_{p}}
\bar{\d} \bxx_{T_1} \cdots \bar{\d} \bxx_{T_{p}},
\end{equation}
where the \emph{singular differentials} $\bar{\d} \bxx_{T}$ are given for any non-negative measurable function $f:(\bR^d)^{|T|} \to \R$ by
\begin{equation*}
\int_{(\bR^d)^{|T|}} f(\bxx_{T}) \bar{\d} \bxx_{T} = \sum_{\tau_1,\ldots,\tau_{|T|}\in \{1,2\}}\int_{\R^d } f((x,\tau_1),\ldots,(x,\tau_{|T|}) ){\d} x.
\end{equation*}

\subsubsection*{Semi-cluster decomposition}
The cumulant measures further decompose into the so-called semi-cluster
measures. These are defined for any disjoint non-empty sets $S,T \subset \{1,\ldots k\}$, and
$A \subset (\bR^{d})^{S}$, $B \subset (\bR^{d})^{T}$  by
\begin{equation*}
U^{S,T}_n(A\times B) = M^{S \cup T}_n(A \times B) - M^{S}_n(A)M^{T}_n(B).
\end{equation*}

Let $S,T$ be a fixed non-trivial partition of
$\{1,\ldots,k\}$. Then, as in~\cite[Lemma~3.2]{raic3}, we have the decomposition
\begin{equation}\label{eq:semi cluster decomposition 1}
c_{n}^{k}
=
\sum_{S'\cup T',T_{1},\ldots,T_{p} \preceq \{1,\ldots k\}}
a_{S', T',T_{1},\ldots,T_{p}}
U^{S',T'}_n M^{T_{1}}_n \dotsm M^{T_{p}}_n,
\end{equation}
where $ a_{S',T',T_{1},\ldots,T_{p}} \in \R$ and the sum runs over all
partitions of $\{1,\ldots,k\}$, such that $S'$ and $T'$ are non-empty
subsets of $S$ and $T$, respectively. 

\subsubsection*{Decomposition with respect to the diagonal}

The maximal separation distance of $\bxx \in (\bR^d)^k$ into two disjoint subsets is defined as 
\begin{equation*}
D(\bxx) = \max_{\{T_1, T_2\} \preceq \{1,\ldots,k\}}
\mathsf{dist}(\xx_{T_2}, \xx_{T_2}). 
\end{equation*}
 For $S,T$ a non-trivial partition of $\{1,\ldots,k\}$, we let
\begin{equation*}
\sigma(S, T)
=
\big\{\bxx = (\bxx_S, \bxx_T) \in \bWn^k:\, D(\bxx) = \mathsf{dist}(\bxx_S, \bxx_T)\big\} \setminus \Delta.
\end{equation*}
Let $S'\subseteq S$ and $T'\subseteq T$. On the set $\sigma(S,T)$
we can use~\eqref{eq:density mixed
moment} to decompose $\d U^{S',T'}$ as
\begin{align}\label{eq:semi cluster decomposition 2}
\d U^{S',T'} &=
\sum_{\{S_1', \dots, S_{q}'\} \preceq S' }  \sum_{\{T_1', \dots, T_{r}'\} \preceq T' } 
\Big(m_n^{(S'_1, \dots, S_{q}',T'_1, \dots, T_{r}')} (\bxx_{S'\cup T'}) \notag \\
&\quad
- m_n^{(S'_1, \dots, S_{q}')}(\bxx_{S'}) m_n^{(T'_1, \dots, T_{r}')}(\bxx_{T'}) \Big) 
\bar{\mathbf{d}} \bxx_{S'\cup T'}, 
\end{align}
where  we have written
$\bar{\mathbf{d}} \bxx_{S'\cup T'} = \bar{\d} \bxx_{S_1'} \cdots \bar{\d} \bxx_{S_{q}'}
\bar{\d} \bxx_{T_1'} \cdots \bar{\d} \bxx_{T_{r}'}$. Note here, that there are terms in the definition \eqref{eq:density mixed moment} of $dM_n^{S'\cup T'}$ corresponding to partitions that do not split into a partition of $S'$ and a partition of $T'$, but these vanish on $\sigma(S',T')$.

As in~\cite[(3.28)]{raic3}, we have the following decomposition for any non-negative measurable function
$\mathbf{f}:(\bR^d)^k \to \R$
\begin{equation} \label{eq:cumulant decomposition}
\int_{(\bR^d)^k} \mathbf{f}(\bxx) dc_{n}^{k}(\bxx)
=
\int_{\Delta}\mathbf{f}(\bxx) \mathrm{d} c_{n}^{k}(\bxx)
+
\sum_{S,T \preceq \{1,\ldots,k\}}
\int_{\sigma(S, T)}  \mathbf{f}(\bxx) \mathrm dc_{n}^{k}(\bxx).
\end{equation}



\subsection{Bounds on mixed $\xi$-moments}\label{sec:xi_weighted_bound}

Let $I_1,I_2 \subseteq [r_0,R]$ be intervals and define score functions 
\begin{equation*}
\xi_{e,n,I_{i}}(x,\X) =  \frac{1}{\rho^{2}} \sum_{y\in \X } \mathds{1}_{\{|x-y|\in I_{i}\}} e_{n}(x,y). 
\end{equation*}
Define the extended score function $\breve{\xi}_{e,n}:\bR^d\times \mathcal{N} \to \R^d$ by 
\begin{equation}\label{augmentedxi}
\breve{\xi}_{e,n}((x, \tau), \X)=
\begin{cases}
\xi_{e,n,I_{1}}(x, \X)&\text{ if $\tau = 1$, }	\\
\xi_{e,n,I_{2}}(x, \X)&\text{ if $\tau = 2$}	
\end{cases}
\end{equation}
and let the associated random  measure $\mu_n$ be as in
\eqref{mun}. In this section, we {establish}
a bound on the mixed $\xi$-moments \eqref{eq:xi-mixed} that is
used in the proof of Theorem \ref{c4bound}. Recall the map $\pi :
\Delta^{T_1,\ldots,T_p} \to \R^p$ that ignores repeated points
and let
\begin{equation*}
W_{n}^{(1,2)} = (W_{n} \times \{1\})^2 \times (W_{n} \times \{2\})^2\subseteq (\bR^d)^4.
\end{equation*}
Similarly, for $T\subseteq\{1,2,3,4\}$, $(W_{n}^{(1,2)})^T$ denotes the coordinates in $T$ of points in $W_n^{(1,2)}$.

\begin{lem}\label{ximomentbound} 
	Let $\PP$ be a stationary point
	process having bounded correlation functions. Let
	$T_1,\ldots,T_p$ be a partition of $T\subseteq
	\{1,\ldots,4\}$ and let $ K_0 >0$.
	Then there is a constant $C>0$ such that for all
	$K\geq K_0$, all disjoint intervals $I_1,I_2 \subseteq
	[r_0,R]$, and all $x_0\in \R^d$, we have 
	\begin{align}\label{intWn}
	&\int_{\Delta^{T_1,\ldots,T_p}\cap \pi^{-1}(W_n\times B_K(x_1)^{p-1})\cap (W_n^{(1,2)})^T} 	m_n^{T_1,\ldots,T_p}(\bxx) \bar{\d}\bxx_{T_1} \ldots \bar{\d}\bxx_{T_p} \\ \nonumber
	& \qquad \qquad \qquad \qquad \qquad \qquad \leq C n K^{d(p-1)}  |I_1|^A|I_2|^B\\
	\label{intBK}
	&\int_{\Delta^{T_1,\ldots,T_p}\cap \pi^{-1}( B_K(x_0)^{p})\cap (W_n^{(1,2)})^T}	m_n^{T_1,\ldots,T_p}(\bxx) \bar{\d} \bxx_{T_1} \ldots \bar{\d} \bxx_{T_p}\\ \nonumber
	& \qquad \qquad \qquad \qquad \qquad \qquad \leq C K^{dp}  |I_1|^A|I_2|^B,
	\end{align}
	where the mixed $\xi$-moments are defined via the score function $\breve{\xi}_{e,n}$ in \eqref{augmentedxi}, and  $A=\mathds{1}_{T\cap\{1,2\}\neq \emptyset}$,  $B=\mathds{1}_{T\cap\{3,4\}\neq \emptyset}$.
\end{lem}

\begin{proof}
	It is enough to show the theorem in the case without edge corrections since the edge correction factors are bounded and hence there is a $C>0$ such that $\xi_{e,n,I_i}\leq C\xi_{e_{1,n},n,I_i}$, $i=1,2$.
	On $(W_n^{(1,2)})^T\cap\Delta^{T_1,\ldots,T_p}$, $m_n^{T_1,\ldots,T_p}$ takes the form
	\begin{equation*}
	m_n^{T_1,\ldots,T_p}(\bxx) = \E_{\xx} \Big[\xi_{I_1}^{l_1}(x_1,\PP_n)\xi_{I_2}^{k_1}(x_1,\PP_n)  \dotsm \xi_{I_1}^{l_p}(x_p,\PP_n)\xi_{I_2}^{k_p}(x_p,\PP_n)\Big]\rho^{(p)} (\xx),
	\end{equation*}
	where $l_i=|T_i\cap\{1,2\}|$ and $k_i=|T_i\cap\{3,4\}|$. Let $V=W_n$. Then, by stationarity, the integral in \eqref{intWn} is bounded by
	\begin{align*}
	{}&\int_{V\times B_K(x_1)^{p-1}} \E_{\xx} \Big[\xi_{I_1}^{l_1}(x_1,\PP)\xi_{I_2}^{k_1}(x_1,\PP) \dotsm \xi_{I_1}^{l_p}(x_{p},\PP)\xi_{I_2}^{k_p}(x_{p},\PP)\Big]\rho^{(p)} (\xx)\d \xx\\
	{}&\leq |V|\int_{ B_K(o)^{p-1}} \E_{o,\mathbf{z}} \Big[\xi_{I_1}^{l_1}(o,\PP)\xi_{I_2}^{k_1}(o,\PP) \dotsm \xi_{I_1}^{l_p}(z_{p-1},\PP)\xi_{I_2}^{k_p}(z_{p-1},\PP)\Big]\rho^{(p)} (o,\mathbf{z})\d \mathbf{z}\\
	&\leq	|V||B_K(o)|^{-1}\int_{B_{2K}(o)^{p}} \E_{\xx} \Big[\xi_{I_1}^{l_1}(x_1,\PP)\xi_{I_2}^{k_1}(x_1,\PP)  \dotsm \xi_{I_1}^{l_p}(x_{p},\PP)\xi_{I_2}^{k_p}(x_{p},\PP)\Big]\rho^{(p)} (\xx)\d \xx.
	\end{align*}
	The same bound obviously holds for \eqref{intBK} when $V=B_K(x_0)$.
	Using the Campbell formula and the definition of $\xi_{I_i}$, we rewrite the integral in the bound as
	\begin{align*}
	&\E\bigg[ \sum_{\xx\in (\PP \cap  B_{2K}(o))_p^{\neq}} \bigg(\sum_{y_1^1\in \PP} \mathds{1}_{|x_1-y_1^1|\in I_1}\bigg)^{l_1} \bigg(\sum_{y_1^2\in \PP} \mathds{1}_{|x_1-y_1^2|\in I_2}\bigg)^{k_1}\\
	&\quad \dotsm \bigg(\sum_{y_p^1\in \PP} \mathds{1}_{|x_p-y_p^1|\in I_1}\bigg)^{l_p} \bigg(\sum_{y_p^2\in \PP} \mathds{1}_{|x_p-y_p^2|\in I_2}\bigg)^{k_p}\bigg].
	\end{align*}
	Multiplying out the sums over $y_i^j$, we obtain a sum over $\boldsymbol{y}\in \PP^{|T|}$, where the points in $\boldsymbol{y}$ are not necessarily all different and could equal some of the points in $\xx$. We now apply the Campbell formula backwards. To do so, we must write the above sum as a linear combination of sums  over different points $x_1,\ldots,x_p,y_1^j,\ldots,y_q^j$. 
	
	To illustrate how each such term is bounded, we consider the case $p=3$, $l_1=l_2=l_3=k_1=1$ and $y_1^1=x_2$, $y_2^1=x_3$ og $y_1^2=x_3$ while $y_3^1$ is different from the points in $\xx$. This yields
	\begin{align*}
	&|B_K(o)|^{-1}\E\bigg[ \sum_{\xx\in (\PP\cap B_{2K}(o))^{\neq}_3} \sum_{y \in \PP} \mathds{1}_{|x_1-x_2|\in I_1} \mathds{1}_{|x_2-x_3|\in I_1} \mathds{1}_{|x_3-y|\in I_1}  \mathds{1}_{|x_1-x_3|\in I_2}\bigg]\\
	&= |B_K(o)|^{-1}\int_{ B_{2K}(o)^3}\int_{\R }\rho^{(4)}(\xx,y) \mathds{1}_{|x_1-x_2|\in I_1} \mathds{1}_{|x_2-x_3|\in I_1} \mathds{1}_{|x_3-y|\in I_1}  \mathds{1}_{|x_2-x_3|\in I_2} \d y \d \xx\\
	&\leq C_1|I_1| |B_K(o)|^{-1}\int_{ B_{2K}(o)^3}\mathds{1}_{|x_1-x_2|\in I_1} \mathds{1}_{|x_2-x_3|\in I_1}  \mathds{1}_{|x_2-x_3|\in I_2}  \d x_1 \d x_2 \ dx_3\\
	&\leq C_1|I_1| |B_K(o)|^{-1}\int_{ B_{2K}(o)^3}  \mathds{1}_{|x_2-x_3|\in I_2}  \d \xx\\
	&\leq C_2|B_{2K}(o)|^2|B_K(o)|^{-1}|I_1||I_2|,
	\end{align*} 
	where we used that $\rho^{(4)}$ is bounded in the first inequality.
	
	It is a straightforward check that all other integrals can be
	bounded similarly. Indeed, the $y_i^j$'s that are free can
	always be integrated out. If the obtained bound involves the
	necessary factors of $|I_1|$ or $|I_2|$ we bound the
	remaining integral by $|B_{2K}(o)^p|$. Otherwise, it is
	convenient to bound one or more of the indicator functions in
	the integral by 1, leaving at most one involving $I_1$ and
	one involving $I_2$. As $I_1$ and $I_2$ are disjoint, a
	product of indicator functions for these two sets must
	involve at least three different points to be non-zero. Thus,
	these can be integrated out one after the other to obtain the
	necessary bound.  
\end{proof}

\subsection{Fast decay of mixed $\xi$-moments}\label{sec:xi_weighted}

In this section, we show an analogue of the fast decay of
correlations for a slightly more general version of the
$\xi$-weighted measures.
Throughout the section, we use the notation
$\xx=(x_1,\ldots,x_k)=(\xx_1,\xx_2)$, where
$\xx_1=(x_1,\ldots,x_p)$ and $\xx_2=(x_{p+1},\ldots,x_k)$. 
	
	Let $\xi_{ij}$,  $i=1,\ldots,k$, $j=1,\ldots L_i$, be score
	functions with a common deterministic radius of stabilization
	$R$ and that there is a constant $\hat{c}>0$ such that for all $i,j$, 
	\begin{equation}\label{score_bound}
	\xi_{ij}(x,\mathcal{X}\cap B_r(x)) 
	\mathds{1}_{\{|\mathcal{X}\cap B_r(x)| = N\}}\leq N\hat{c}.
	\end{equation}
	Define the associated $\xi$-weighted moments 
	\begin{equation*}
	m^{L_1,\ldots,L_k}_n(\xx ) = 
	\E_{\xx} \bigg[ \prod_{l_1=1}^{L_1} 
	\xi_{1l_1}(x_1,\mathcal{P}_n) \dotsm 
	\prod_{l_k=1}^{L_k} \xi_{kl_k}(x_k,\mathcal{P}_n) \bigg]
	\rho^{(k)}(\xx).
	\end{equation*}

\begin{lem}\label{lem:fast_xi} 
	Let $\PP$ be a point process
	having fast decay of correlations and satisfying Condition
	{\bf (M)} or a Gibbs point process of class
	$\mathbf{\Psi}^*$. Let $\xi_{ij}$ be score functions for
	$i=1,\ldots,k$, $j=1,\ldots L_i$ with a common deterministic
	radius of stabilization $R$ and all satisfying the common
	bound \eqref{score_bound}. Then there exists a constant $C>0$
	and a fast decreasing function $\omega $ depending only on
	the score functions via the constant $\hat{c}$ in
	\eqref{score_bound}, the common radius of stablization $R$,
	and the numbers $k$, $L_1,\ldots,L_k$  such that for all $n$
	and almost all $\xx =(\xx_1,\xx_2)$,
	$\xx_1=(x_1,\ldots,x_p)$, $\xx_2=(x_{p+1},\ldots,x_k)$,
	\begin{itemize}
	\item[(i)] $|m^{L_1,\ldots,L_k}_n(\xx ) | \leq C$
	\item[(ii)] $\big|m^{L_1,\ldots,L_k}_n(\xx ) - m^{L_1,\ldots,L_p}_n(\xx_1 )m^{L_{p+1},\ldots,L_k}_n(\xx_2 ) \big|\leq \omega(\dist(\xx_1,\xx_2 ))$.
	\end{itemize}
\end{lem}

\begin{proof}
First we consider a point process having fast decay of
correlations and satisfying Condition {\bf (M)}. By the
H\"{o}lder inequality,
\begin{align*}
	m_n^{L_1,\ldots,L_k}(\xx ) {}&= 
	\E_{\xx } \bigg[ \prod_{l_1=1}^{L_1} \xi_{1l_1}(x_1,\PP_n) 
	\dotsm \prod_{l_k=1}^{L_k} \xi_{kl_k}(x_k,\PP_n) \bigg]
	\rho^{(k)}(\xx)\\
	&\leq \hat{c}^{L}\rho^{(k)}(\xx)
	\E_{\xx } \bigg[  \prod_{j=1}^{k} \PP(B_R(x_j))^{L_j} \bigg]\\
	&\leq   \hat{c}^{L}\rho^{(k)}(\xx)  \prod_{j=1}^{k}\E_{\xx } \bigg[  \PP(B_R(x_j))^{kL_j} \bigg]^{1/k},
\end{align*}
where $L=L_1 + \dotsm + L_k$. By Condition {\bf (M)} and
boundedness of $\rho^{(k)}(\xx)$, this is bounded, which
shows (i). 

In  \cite[Thm. 1.11]{yogesh}, (ii) was shown for the case 
\begin{equation*}
m_n^{L_1,\ldots,L_k}(\xx ) = \E_{\xx} \Big[\xi(x_1,\PP_n)^{L_1} 
  \dotsm  \xi(x_k,\PP_n)^{L_k}\Big]\rho^{(k)} (\xx),
\end{equation*}
where $\xi$ is a score function that does not depend on $n$,
satisfies \eqref{score_bound}, and is invariant under
translation, i.e.\ $\xi(x,\mathcal{X}) = \xi(x+y,\mathcal{X}
+y)$.
 
Our case is different in several ways. First of all, not all
score functions are identical. However, this is not a problem for
the proof of \cite[Thm. 1.11]{yogesh}, since the only properties
needed in the proof was a factorization \cite[(3.21)]{yogesh}
of the form
	\begin{equation*}
\prod_{j=1}^k\prod_{l_j=1}^{L_j} \xi_{jl_j}(x_j,\mathcal{X}) = \bigg(\prod_{j=1}^p\prod_{l_j=1}^{L_j} \xi_{jl_j}(x_j,\mathcal{X})\bigg)\bigg(\prod_{j=p+1}^k\prod_{l_j=1}^{L_j} \xi_{jl_j}(x_j,\mathcal{X})\bigg)
\end{equation*}
 and a common bound of the form \eqref{score_bound} for all the
 involved score functions. The proof yields a bound of the form
 (ii) where the fast decreasing function $\omega$ depends only on
 the score functions involved via the constant $\hat{c}$.
 
Moreover, the involved score function was assumed translation
invariant in \cite{yogesh}, which is not assumed here. However,
the proof of \cite[Thm. 1.11]{yogesh} does not use translation
invariance, since the factorial moment expansion \cite{fme}, on
which the proof was built, does not require translation
invariance. As long as the bound \eqref{score_bound} holds for
all $(x,\mathcal{X})$, the proof still goes through. 

Next, we consider the Gibbs case. Assume that $\PP$ is
constructed by thinning a free birth-death process $\gamma(t)$ of
intensity $\tau$ as described in Appendix \ref{app:thinning}. Let
$\mathcal{Q}$ be the Poisson process $\gamma(0)$ and
$\PP=\gamma^{\Psi}(0)\subseteq \gamma(0) = Q$. For any
$A\subseteq \R^{dk}$, 
\begin{align*}
\int_{A} m_n^{L_1,\ldots,L_k}(\xx )  \d \xx {}& = \E\bigg[\sum_{\xx \in \PP_n^{\neq}} \mathds{1}_{A}(\xx)\prod_{l_j=1}^{L_j} \xi_{jl_j}(x_j,\PP_n)  \bigg]\\
&\leq \hat{c}^{L}  \E\bigg[\sum_{\xx \in \PP_n^{\neq}} \mathds{1}_{A}(\xx ) \PP(B_{R}(x_j))^{L_j}   \bigg]\\
 &\leq\hat{c}^{L}  \E\bigg[\sum_{\xx \in \mathcal{Q}_n^{\neq}} \mathds{1}_{A}(\xx) \mathcal{Q}(B_R(x_j))^{L_j}\bigg]\\
 &\leq |A| C',
\end{align*}
where  the last inequality follows from (i) that has been proved
 for point processes satisfying Condition \textbf{(M)} and thus
 for Poisson point processes. It follows that
 $m_n^{L_1,\ldots,L_k}(\xx,\PP_n)$ must be bounded by $C'$ for
 almost all $\xx$.

To show (ii) we fix $\xx$, let $A = \prod_{j=1}^k A^j$ where $x_j\in A^j \subseteq B_1(x_j)$, and let $s=\tfrac{1}{2}\dist(\xx_1,\xx_2 ) - R-1 $. Let $A_1=\prod_{j=1}^p A^j$ and $A_2=\prod_{j=p+1}^k A^j$. For $s>0$,
\begin{align}\nonumber
\int_{A} m_n^{L_1,\ldots,L_k}(\mathbf{y})  \d \mathbf{y} {}& = \E\bigg[\sum_{\mathbf{y} \in (\PP_n)_k^{\neq}} \mathds{1}_{A}(\mathbf{y})\prod_{j=1}^k\prod_{l_j=1}^{L_j} \xi_{jl_j}(y_j,\PP_n) \bigg]\\
& = \E(f(\PP_n)g(\PP_n)), \label{campbell_omskrivning}
\end{align}
where
\begin{align*}
f(\mathcal{X}){}& = \sum_{\mathbf{y}_1 \in \mathcal{X}_p^{\neq}} \mathds{1}_{A_1}(\mathbf{y}_1)\prod_{j=1}^p \prod_{l_j=1}^{L_j} \xi_{jl_j}(y_j,\mathcal{X}) \\
g(\mathcal{X}) {}& = \sum_{\mathbf{y}_2\in \mathcal{X}_{k-p}^{\neq}} \mathds{1}_{A_2}(\mathbf{y}_2)\prod_{j=p+1}^{k}\prod_{l_j=1}^{L_j} \xi_{jl_j}(y_j,\mathcal{X}).
\end{align*}

Let $E_1$ and $E_2$ be the events that all ancestors of points in $A_1'= (\bigcup_{j=1}^p A^j)\oplus B_R(o)$  and  $A_2'= (\bigcup_{j=p+1}^k A^j)\oplus B_R(o)$ are within distance $s-r^{\Psi}$ from  $A_1'$ and $A_2'$, respectively. The events $E_i$ depend only on $\gamma(t)$ restricted to $(A_i'\oplus B_s(o)) \times \R$, which are disjoint sets. On $E_i$, $\PP\cap A_i'$ depends only on $\gamma(t) $ restricted to $(A_i'\oplus B_s(o)) \times \R$. 

Then, since $\mathds{1}_{E_i} = 1 - \mathds{1}_{E_i^c}$, we may write
\begin{align*}
&\E[f(\mathcal{P}_n)g(\PP_n)] = \E[f(\mathcal{P}_n)\mathds{1}_{E_1}] \E[ g(\PP_n)\mathds{1}_{E_2}] +  \E [f(\mathcal{P}_n )g(\PP_n)\mathds{1}_{(E_1\cap E_2)^c}] \\
& = \E[f(\mathcal{P}_n)]\E [ g(\PP_n)] - \E[f(\mathcal{P}_n)\mathds{1}_{E_1^c}]\E [ g(\PP_n)] - \E[f(\mathcal{P}_n)]\E [ g(\PP_n)\mathds{1}_{E_2^c}]\\
& +\E[f(\mathcal{P}_n)\mathds{1}_{E_1^c}]\E [ g(\PP_n)\mathds{1}_{E_2^c}] + \E [f(\mathcal{P}_n )g(\PP_n)\mathds{1}_{(E_1\cap E_2)^c}]. 
\end{align*}
The first equality follows from independence between disjoint parts of $\gamma(t)$. We need  to bound the last four terms. Since the argument is the same in all four cases, we just consider $\E[f(\mathcal{P}_n)\mathds{1}_{E_1^c}]\E[ g(\PP_n)]$. By (i), $\E [ g(\PP_n)] \leq C  |A_2| $. By the Cauchy-Schwarz inequality,
\begin{align*}
\E[f({\PP_n})\mathds{1}_{E_1^c} ]\leq \E[f({\PP_n})^2]^{1/2}\mathbb{P}(E_1^c )^{1/2} \leq C_1  |A_1|\exp(-s/C_2).
\end{align*}
The latter bound follows from \eqref{eq:gibbs_ancestor} in Appendix \ref{app:thinning} and because
\begin{align*}
 &\E[f({\PP_n})^2] \leq \hat{c}^{2L}\E  \bigg[ \bigg(\sum_{\mathbf{y}_1 \in (\mathcal{Q}_n)_p^{\neq}} \mathds{1}_{A_1}(\mathbf{y}_1) \mathcal{Q}_n(B_R(y_j))^{L_j}\bigg)^2 \bigg]\\
 &\leq \hat{c}^{2L} \E\bigg[\sum_{\{j_1,\ldots,j_l\} \subseteq \{1,\ldots,p\}} c_l \sum_{\tilde{\mathbf{y}}\in (Q_n)_{p+l}^{\neq} }  \mathds{1}_{A_1}(\mathbf{y}_1)  \prod_{i=1}^{l} \mathds{1}_{A^{j_i}}(y_{j+i}) \prod_{j=1}^{p+l}  \mathcal{Q}_n (B_R(y_j))^{2L}\bigg]  \\
 &\leq \hat{c}^{2L}  \sum_{\{j_1,\ldots,j_l\} \subseteq \{1,\ldots,p\}}c_l\tau^{p+l}\int_{ A_1 \times \prod_{i=1}^{l} A^{j_i} }\E \left[ \prod_{j=1}^{p+l}  (\mathcal{Q} \cup \tilde{\mathbf{y}} ) (B_R(y_j))^{2L}\right] \d \tilde{\mathbf{y}}\\  
 &\leq C|A_1|,
\end{align*}
where $\tilde{\mathbf{y}} = (\mathbf{y}_1,y_{p+1},\ldots,y_{p+l})$. The last inequality follows from the H\"{o}lder inequality and boundedness of moments for the Poisson process.

The above shows that
\begin{align*}
&\bigg|\int_{A} ( m_n^{L_1,\ldots,L_k}(\mathbf{y},\PP_n) - m_n^{L_1,\ldots,L_p}(\mathbf{y}_1,\PP_n)m_n^{L_{p+1},\ldots,L_k}(\mathbf{y}_2 ,\PP_n) )\d \mathbf{y} \bigg|\\
&\quad \leq |A|C\exp(-\dist(\xx_1,\xx_2)/c).
\end{align*}
Finally, (ii) is shown by the Lebegue
differentiation theorem when $A$ tends to $\xx$.

\end{proof}

\subsection{Proof of Lemma \ref{c4bound}}\label{sec:c4 lem proof}

\begin{proof}[Proof of Lemma \ref{c4bound}]
According to~\eqref{cummesdef} and~\eqref{eq:cumulant decomposition} with  $k=4$ and  $\mathbf{f}= \mathds{1}_{W_{n}^{(1,2)}}$, we have
\begin{align*}
  c^4(n\bar{K}_{e,n}(I_1), & n\bar{K}_{e,n}(I_1), n\bar{K}_{e,n}(I_2),n\bar{K}_{e,n}(I_2)) \\
  &= c^4(\mu_n(W_n\times \{1\}),\mu_n(W_n\times \{1\}),\mu_n(W_n\times \{2\}),\mu_n(W_n\times \{2\}))\\
  & =\int_{(\bR^d)^4} \mathds{1}_{W_n^{(1,2)}} \d c_n^4\\
                     &=
                       \int_{\Delta} \mathds{1}_{W_{n}^{(1,2)}}(\bxx) \mathrm{d} c_{n}^{4}(\bxx) 
                       + 
                       \sum_{S,T \preceq \{1,\ldots,4\}}
                       \int_{\sigma(S, T)}  \mathds{1}_{W_{n}^{(1,2)}}(\bxx) \mathrm 
                       dc_{n}^{4}(\bxx).
\end{align*}

We first consider the diagonal term.
  In~\eqref{eq:cumulant xi n}, 
 only the term $p=1$ contributes on $\Delta$, i.e. $(c_n^4)_{| \Delta} = (M^4_n)_{|\Delta}$. Also in \eqref{eq:density mixed moment}, only the term $p=1$ contributes on $\Delta$ such that $(c_n^4)_{| \Delta} = m_n^{\{1,2,3,4\}} \d \bxx_{\{1,2,3,4\}}$.  Inserting the definition of the score function \eqref{augmentedxi}, we get
\begin{align*}
\int_{\Delta} \mathds{1}_{W_{n}^{(1,2)}}(\bxx) \mathrm{d}
c_{n}^{4}(\bxx)
&=\int_{W_n}  \E_{x}\big[ \xi_{e,n,I_1}(x,\PP_n)^2   
\xi_{e,n,I_{2}}(x, \PPn)^2   \big]\rho\d x \\
&\leq nC \E_{o}\big[ \xi_{e_{1,n},n,I_1}(o,\PP)^2   
\xi_{e_{1,n},n,I_{2}}(o, \PP)^2   \big] \rho\\
&\leq nC'|I_1||I_2|.
\end{align*}
The first inequality follows because the edge correction 
factors are bounded and hence there is a $C>0$ such that $\xi_{e,n,I_i}\leq C\xi_{e_{1,n},n,I_i}$, $i=1,2$, and the uncorrected score function $\xi_{e_{1,n},n,I_i}$ is translation invariant.
The last inequality used the special case $T=\{1,2,3,4\}$ and $p=1$ in Lemma~\ref{ximomentbound}.

It remains to bound the terms
\begin{equation*}
 \int_{\sigma(S, T)} \mathds{1}_{W_{n}^{(1,2)}}(\bxx) \mathrm
dc_{n}^{4}(\bxx).
\end{equation*}
To that end, fix a non-trivial partition
 $\{S,T\}$ of $\{1,2,3,4\}$. Set $K=(|I_1||I_2|)^{-1/(12d)}$ and
 let $D_{K}=\{ \bxx \in W_{n}^{(1,2)} \cap \sigma(S,T) : D(\bxx)
 > K \}$. We study the integral over $D_{K}$ and $D_{K}^{c}$
 separately.

We first consider the integral over $D_K$. By~\eqref{eq:semi
  cluster decomposition 1}-\eqref{eq:semi cluster decomposition
  2}, this is bounded by integrals of the form
\begin{equation*}
  \int_{D_K}
  \big(m_n^{S'_1, \dots, S_{q}',T'_1, \dots, T_{r}'}(\bxx_{S'\cup T'})
  -
  m_n^{S'_1, \dots, S_{q}'}(\bxx_{S'}) m_n^{T'_1, \dots, T_{r}'}(\bxx_{T'})  \big) \notag 
  \bar{\mathbf{d}} \bxx_{S'\cup T'}
  \d M^{T_{1}} \ldots \d M^{T_{p}},
\end{equation*}
where $S'\subseteq S$, $T'\subseteq T$, $S'_1, \dots, S_{q}'$ is
a partition of $S'$, $T'_1, \dots, T_{r}'$ is a partition of
$T'$, and $S'\cup T', T_{1},\ldots T_{p}$ is a partition of
$\{1,2,3,4\}$.

To bound this, we obtain from (ii) in Lemma
\ref{lem:fast_xi}  a fast decreasing function
$\omega$ not depending on $n$, $I_1$, and $I_2$ such that
\begin{equation*}
   \big|m_n^{S'_1, \dots, S_{q}',T'_1, \dots, T_{r}'}(\bxx_{S'\cup T'})
  -
  m_n^{S'_1, \dots, S_{q}'}(\bxx_{S'}) m_n^{T'_1, \dots, T_{r}'}(\bxx_{T'}) \big|
  \leq
   \omega \left( D(\bxx_{S'\cup T'}) \right).
\end{equation*}

Let $m=12d$. 
By Definition~\ref{def:expdecay}, there exists a constant
$c'>0$ such that $\omega(t) \leq {c'}{t^{-m}}$ so that
\begin{equation*}
  \int_{D_{K}}
  \omega\left( D(\bxx_{S'\cup T'})  \right)
  \bar{\mathbf{d}} \bxx_{S'\cup T'}
  \d M^{T_{1}} \dotsm \d M^{T_{p}}
  \leq
  c'
  \int_{D_{K}}
  \frac{\bar{\mathbf{d}} \bxx_{S'\cup T'}}{D(\bxx_{S'\cup T'})^{m}} 
  \d M^{T_{1}} \dotsm \d M^{T_{p}}.
\end{equation*}

Since we are on $\sigma(S,T)$, $D(\bxx_{S'\cup T'}) \geq \dist(\bxx_{S'},\bxx_{T'})\geq \dist(\bxx_S,\bxx_T)$. Moreover, noting that $1\leq |S|,|T|\leq 3$, all points in $\bxx_{S}$ and $\bxx_{T}$, respectively, must be within distance $2\dist(\bxx_{S},\bxx_{T} ) $ of each other.

Letting $\delta(\bxx)$ denote the maximal pairwise distance between points in $\bxx$, we get 
\begin{align}\label{eq:ugly_int}
&\int_{D_{K}}
\frac{\bar{\mathbf{d}} \bxx_{S'\cup T'}}{D(\bxx_{S'\cup T'})^{m}} 
\d M^{T_{1}} \dotsm \d M^{T_{p}}\\
&\quad =
\int_{D_{K}}
\frac{1_{\{\delta(\bxx_S),\delta(\bxx_T)\leq 2 \dist(\bxx_S,\bxx_T)\}}}{\dist(\bxx_{S},\bxx_{ T})^{m}} \bar{\mathbf{d}} \bxx_{S'\cup T' }
\d M^{T_{1}} \dotsm \d M^{T_{p}}.\nonumber
\end{align}
 We now recall the decomposition \eqref{eq:density mixed moment}
\begin{equation}
\d M^{T_i}
=
\sum_{T_1'', \dots, T_{p'}'' \preceq T_i }
m_n^{T''_1, \dots, T_{p'}''} ( \bxx )
\bar{\d} \bxx_{T_1''} \cdots \bar{\d} \bxx_{T_{p'}''}. 
\end{equation}
We note that on $\sigma(S,T)$ only terms with all $T_j''\subseteq S$ or $T_j''\subseteq T$ contribute since the measure $\bar{\d} \bxx_{T_j''}$ is concentrated on the diagonal of $\bR^{|T_j''|}$.
Moreover, we know from Lemma \ref{lem:fast_xi} that all $m_n^{T_1'',\ldots,T_s''}$ with $\sum_i |T_i''|\leq 3$  are bounded by a constant.

Letting  $\boldsymbol{x}=\xx_{S'}=(x_1,\ldots,x_q)$ , $\boldsymbol{y}=\xx_{T'}=(y_1,\ldots,y_r)$, $\boldsymbol{z}=(z_1,\ldots,z_s)$, and $\boldsymbol{w}=(w_1,\ldots,w_t)$, $2\leq q+r+s+t\leq 4$, the integral \eqref{eq:ugly_int} is thus bounded by terms of the form  
\begin{align*}
C\int_{W_n^{t+s}}\int_{W_n^{r+q}}
\frac{1_{\{\dist((\boldsymbol{x},\boldsymbol{z}),(\boldsymbol{y},\boldsymbol{w}))> K\}}1_{\{\delta(\boldsymbol{x},\boldsymbol{z}),\delta(\boldsymbol{y},\boldsymbol{w})\leq 2\dist((\boldsymbol{x},\boldsymbol{z}),(\boldsymbol{y},\boldsymbol{w}) )\} }}{\dist((\boldsymbol{x},\boldsymbol{z}),(\boldsymbol{y},\boldsymbol{w}))^m} \d \boldsymbol{x} \d \boldsymbol{y}
\d \boldsymbol{z} \d \boldsymbol{w}.
\end{align*}
On the set where $\dist((\boldsymbol{x},\boldsymbol{z}),(\boldsymbol{y},\boldsymbol{w}))$ is attained  as $|x_1 - y_1|$, the integral is bounded by
\begin{align*}
&C\int_{W_n^{t+s+r+q}}
\frac{1_{\{|x_1-y_1|> K\}}1_{\{\max\{|x_1-x_i|,|x_1-z_j|,|y_1-y_k|,|y_1-w_l|\}\leq 2 |x_1-y_1|\} }}{|x_1-y_1|^m} \d \boldsymbol{x} \d \boldsymbol{y}
\d \boldsymbol{z} \d \boldsymbol{w}\\
&\leq C'\int_{W_n^{2}}
\frac{1_{\{|x_1-y_1|> K\}}}{|x_1-y_1|^{m-d(s+t+q+r-2)}} \d x_1
\d y_1 \\
&	\leq C''|W_n| K^{d(3-s-t-q-r)-m}\\
&\leq C''n(|I_1||I_2||)^{3/4},
\end{align*}
where the maximum is taken over $i=1,\ldots,q, j=1,\ldots, s,k=1,\ldots,r,
l=1,\ldots,t$.
The remaining integrals are bounded similarly. 

On $D_K^c$, each of the four points in $\bxx\in (\bR^d)^4$ must be within 
distance $3K$ from each other. Thus,
\begin{equation*}
c^4_n(W_n^{(1,2)}\cap D_K^c) \leq \int_{W_n^{(1,2)}} \mathds{1}_{B_{3K}(x_1)^3}(x_2,x_3,x_4) c_n^4(\d \bxx).
\end{equation*}
Using the decomposition \eqref{eq:cumulant xi n} and
\eqref{eq:density mixed moment} and stationarity, we obtain a
factorization of this integral into factors of the
same form as the ones in Lemma
\ref{ximomentbound}. This yields the bound  
\begin{equation*}
c^4_n(W_n^{1,2}\cap D_K^c) \leq C n K^{3d} |I_1||I_2| =  C n ( |I_1||I_2|)^{3/4}.
\end{equation*}

\end{proof}


\begin{thebibliography}{22}
	\providecommand{\natexlab}[1]{#1}
	\providecommand{\url}[1]{\texttt{#1}}
	\expandafter\ifx\csname urlstyle\endcsname\relax
	  \providecommand{\doi}[1]{doi: #1}\else
	  \providecommand{\doi}{doi: \begingroup \urlstyle{rm}\Url}\fi
	
	\bibitem[Baddeley et~al.(2015)Baddeley, Rubak, and
	  Turner]{Baddeley:Rubak:Wolf:15}
	A.~Baddeley, E.~Rubak, and R.~Turner.
	\newblock \emph{Spatial Point Patterns: Methodology and Applications with {R}}.
	\newblock Chapman and Hall/CRC Press, 2015.
	
	\bibitem[Biscio et~al.(2020)Biscio, Chenavier, Hirsch, and Svane]{bchs20}
	C.~A.~N. Biscio, N.~Chenavier, C.~Hirsch, and A.~M. Svane.
	\newblock Testing goodness of fit for point processes via topological data
	  analysis.
	\newblock \emph{Electronic Journal of Statistics}, 14\penalty0 (1):\penalty0
	  1024--1074, 2020.
	
	\bibitem[B{\l}aszczyszyn et~al.(1997)B{\l}aszczyszyn, Merzbach, and
	  Schmidt]{fme}
	B.~B{\l}aszczyszyn, E.~Merzbach, and V.~Schmidt.
	\newblock A note on expansion for functionals of spatial marked point
	  processes.
	\newblock \emph{Statististics and Probability Letters}, 36\penalty0
	  (3):\penalty0 299--306, 1997.
	
	\bibitem[B{\l}aszczyszyn et~al.(2019)B{\l}aszczyszyn, Yogeshwaran, and
	  Yukich]{yogesh}
	B.~B{\l}aszczyszyn, D.~Yogeshwaran, and J.~E. Yukich.
	\newblock Limit theory for geometric statistics of point processes having fast
	  decay of correlations.
	\newblock \emph{Annals of Probability}, 47\penalty0 (2):\penalty0 835--895,
	  2019.
	
	\bibitem[Chiu et~al.(2013)Chiu, Stoyan, Kendall, and Mecke]{chiu}
	S.~N. Chiu, D.~Stoyan, W.~S. Kendall, and J.~Mecke.
	\newblock \emph{Stochastic Geometry and Its Applications}.
	\newblock John Wiley \& Sons, 2013.
	
	\bibitem[Dereudre(2018)]{dereudre}
	D.~Dereudre.
	\newblock Introduction to the theory of gibbs point processes.
	\newblock \emph{arXiv:1701.08105v2}, 2018.
	
	\bibitem[Eichelsbacher et~al.(2015)Eichelsbacher, Rai\v{c}, and
	  Schreiber]{raic3}
	P.~Eichelsbacher, M.~Rai\v{c}, and T.~Schreiber.
	\newblock Moderate deviations for stabilizing functionals in geometric
	  probability.
	\newblock \emph{Annales de l'institut Henri Poincar\'e, Probabilit\'e et
	  Statistiques}, 51\penalty0 (1):\penalty0 89--128, 2015.
	
	\bibitem[Fernandez et~al.(1998)Fernandez, Ferrari, and Garcia]{fernandez}
	R.~Fernandez, P.~A. Ferrari, and N.~L. Garcia.
	\newblock Measures on contour, polymer or animal models. a probabilistic
	  approach.
	\newblock \emph{Markov Processes and Related Fields}, 2:\penalty0 479--497,
	  1998.
	
	\bibitem[Ferrari et~al.(2002)Ferrari, Fernandez, and Garcia]{ferrari}
	P.~A. Ferrari, R.~Fernandez, and N.~L. Garcia.
	\newblock Perfect simulation for interacting point processes, loss networks and
	  ising models.
	\newblock \emph{Stochastic Processes and Their Applications}, 102:\penalty0
	  63--88, 2002.
	
	\bibitem[Harris(1963)]{harris}
	T.~E. Harris.
	\newblock \emph{The Theory of Branching Processes}.
	\newblock Springer-Verlag, Berlin, 1963.
	
	\bibitem[Heinrich(1991)]{heinrich91}
	L.~Heinrich.
	\newblock Goodness-of-fit tests for the second moment funciton of a stationary
	  multidimensional {P}oisson process.
	\newblock \emph{Statistics}, 22\penalty0 (2):\penalty0 245--268, 1991.
	
	\bibitem[Heinrich(2015)]{kclt}
	L.~Heinrich.
	\newblock Gaussian limits of empirical multiparameter {$K$}-functions of
	  homogeneous {P}oisson processes and tests for complete spatial randomness.
	\newblock \emph{Lithuanian Mathematical Journal}, 55\penalty0 (1):\penalty0
	  72--90, 2015.
	
	\bibitem[Heinrich and Schmidt(1985)]{heinrichschmidt}
	L.~Heinrich and V.~Schmidt.
	\newblock Normal convergence of multidimensional shot noise and rates of this
	  convergence.
	\newblock \emph{Advances in Applied Probability}, 17\penalty0 (4):\penalty0
	  709--730, 1985.
	
	\bibitem[Kallenberg(2002)]{kallenberg}
	O.~Kallenberg.
	\newblock \emph{Foundations of Modern Probability}.
	\newblock Springer, New York, second edition, 2002.
	
	\bibitem[Kendall and M\o{}ller(2000)]{kendall}
	W.~S. Kendall and J.~M\o{}ller.
	\newblock Perfect simulation using dominating processes on ordered spaces, with
	  application to locally stable point processes.
	\newblock \emph{Advances in Applied Probability}, 32:\penalty0 844--865, 2000.
	
	\bibitem[Krebs and Hirsch(2021)]{hirsch}
	J.~Krebs and C.~Hirsch.
	\newblock Functional central limit theorems for persistent betti numbers on
	  cylindrical networks.
	\newblock \emph{Scandinavian Journal of Statistics}, 2021.
	
	\bibitem[Myllymäki et~al.(2017)Myllymäki, Mrkvička, Grabarnik, Seijo, and
	  Hahn]{globalenv17}
	M.~Myllymäki, T.~Mrkvička, P.~Grabarnik, H.~Seijo, and U.~Hahn.
	\newblock Global envelope tests for spatial processes.
	\newblock \emph{Journal of the Royal Statistical Society: Series B (Statistical
	  Methodology)}, 79\penalty0 (2):\penalty0 381--404, 2017.
	
	\bibitem[Ripley(1977)]{ripley2}
	B.~D. Ripley.
	\newblock Modelling spatial patterns.
	\newblock \emph{Journal of the Royal Statistical Society: Series B (Statistical
	  Methodology)}, 239:\penalty0 172--212, 1977.
	
	\bibitem[Ripley(1988)]{ripley}
	B.~D. Ripley.
	\newblock \emph{Statistical Inference for Spatial Processes}.
	\newblock Cambridge University Press, 1988.
	
	\bibitem[Schreiber and Yukich(2013)]{gibbs_limit}
	T.~Schreiber and J.~E. Yukich.
	\newblock Limit theorems for geometric functionals of gibbs point processesl.
	\newblock \emph{Annales de l'institut Henri Poincar\'e, Probabilit\'e et
	  Statistiques}, 49\penalty0 (4):\penalty0 1158--1182, 2013.
	
	\bibitem[Thomas and Owada(2021)]{owada21}
	A.~M. Thomas and T.~Owada.
	\newblock Functional limit theorems for the euler characteristic process in the
	  critical regime.
	\newblock \emph{Advances in Applied Probability}, 53:\penalty0 57--80, 2021.
	
	\bibitem[Xia and Yukich(2015)]{gibbsCLT}
	A.~Xia and J.~E. Yukich.
	\newblock Normal approximation for statistics of {G}ibbsian input in geometric
	  probability.
	\newblock \emph{Advances in Applied Probability}, 47\penalty0 (4):\penalty0
	  934--972, 2015.
	
	\end{thebibliography}
\bibliographystyle{plainnat}


\appendix
\section{Results on conditional variances}\label{app:condvar}
This appendix contains two lemmas referred to in the proof of
Proposition~\ref{prop:variance lower bound}.
\begin{lem}\label{23Lem}
	Let $Y$ be a square integrable random variable and let $\sigma$ be a  $\sigma$-algebra. Let $I_1, I_2$
	be Borel subsets of $\R$. Then,
	\begin{equation*}
	\Var(Y|\sigma) \geq \frac{1}{4} \min_{i \in \{1,2\}} \P(Y \in I_i|\sigma)\inf_{x_1 \in I_1, x_2 \in I_2}|x_1 - x_2|^2.
	\end{equation*}
\end{lem}
This is a conditional version of \cite[Lemma 2.3]{gibbsCLT}. 
The second lemma contains some results on conditional variances. The proofs are included for completeness. 

\begin{lem}\label{lemma:ineq conditional variance}
	For any square integrable random variables $X$ and $Y$ and $\sigma$-algebras
	$\sigma_{1}$, $\sigma_{2}$ such that $\sigma_{1} \subset \sigma_{2}$, the following holds:
	\begin{enumerate}
		\item \label{varineq} 
		\begin{equation*}
		\E(\Var(X|\sigma_{2})) \leq  \E(\Var(X|\sigma_{1})).
		\end{equation*}
		\item \label{vareq} If $Y$ is measurable with respect to $\sigma_1$, then
		\begin{equation*}
		\Var(X+Y|\sigma_{1}) =  \Var(X|\sigma_{1}).
		\end{equation*}
	\end{enumerate}	
\end{lem}

\begin{proof}
	First we prove \ref{varineq}. For $i=1,2$, we let $X_{i}=\E(X|\sigma_{i})$. Then, the law of total variance applied to $X_2$ yields
	\begin{equation*}
	\Var(X_2) = \E \Var(X_2 | \sigma_1) + \Var (\E(X_2|\sigma_1)) = \E \Var(X_2 | \sigma_1) + \Var (X_1).
	\end{equation*}
	It follows, that $\Var(X_1)\leq \Var(X_2)$. The result now follows from the law of total variance applied to $X$:
	\begin{equation*}
	\Var(X) = \E \Var(X | \sigma_1) + \Var (X_1) = \E \Var(X | \sigma_2) + \Var (X_2) .
	\end{equation*}
	
	To prove \ref{vareq}, we compute 
	\begin{align*}
	\Var(X+Y |\sigma_1){}&= \E( (X+Y - \E(X+Y|\sigma_1))^2  | \sigma_1)\\
	&= \E ( (X - \E(X|\sigma_1))^2 |\sigma_1)\\
	& =\Var(X|\sigma_1).
	\end{align*}
\end{proof}

\section{Appendix on Gibbs point processes}\label{app:thinning}

The infinite volume Gibbs point process $\PP$ of class $\mathbf{\Psi}^*$ 
with energy functional $\Psi$,  activity $\tau$ satisfying \eqref{eq:tau_condition}, and inverse temperature $\beta$ was introduced in Section \ref{sec:gibbs} as the point process satisfying \eqref{eq:gibbs_def}. Here we first relate the infinite volume Gibbs point process to the finite volume case and then we recall the perfect simulation construction from \cite{gibbs_limit}.

Consider a finite window $D$. The finite volume Gibbs process $\PP_{D}$ on $D$ is absolutely continuous with respect to an intensity $\tau$ Poisson point process $Q $ on $D$ with density
\begin{equation*}
\X \mapsto \frac{\exp(-\beta \Psi(\X ))}{ \E [\exp(-\beta \Psi(Q\cap D ))]}.
\end{equation*}

Define the stationarization $\bar{\PP}_{W_n}$ of $\PP_{W_n}$ to be the point process satisfying
\begin{equation*}
\int f(\X) \bar{\PP}_{W_n}(\d \X) =\frac{1}{|W_n|} \int_{W_n} \int f(\X + u) {\PP}_{W_n}(\d \X)\d u. 
\end{equation*} 
Since the infinite volume Gibbs process is unique when \eqref{eq:tau_condition} holds, the sequence $(\bar{\PP}_{W_n})_{n\geq 1}$ converges in the local convergence topology to an infinite volume Gibbs process, see \cite{dereudre}. The construction below yields a construction of $({\PP}_{W_n})_{n\geq 1}$ that converges to the infinite Gibbs process.

\subsection{Perfect simulation of Gibbs point processes}

We now recall a construction of the infinite Gibbs process $\PP$
by thinning a free birth-death process $\gamma(t)$ on $\R^d\times
\R$ of birth intensity $\tau$ and death intensity 1. This
construction was introduced in \cite{ferrari} for the area
interaction process and generalized to a larger class of Gibbs
point processes in \cite{gibbs_limit}. Since we only
consider energy functionals of finite range, we only need a
simplified version of the construction presented
in~\cite{gibbs_limit} which we review below.

We start by constructing a birth-death process
$\gamma_D^{\Psi}(t)$ on $D\times \R$ where $D$ is a bounded
domain. Let $\gamma_D(t)$ be the restriction of the free
birth-death process $\gamma(t)$ to $D\times \R$. Every time a
point in $\gamma_D(t)$ is born, we accept it with a certain
probability. We define $\gamma^{\Psi}_D(t)$ as the
birth-death process formed only by the points that are accepted when they are born.
 The probability of
accepting a point $x\in D$ born at time $t$ is $\exp(-\beta
\Delta^{\Psi}(x,\gamma^{\Psi}_D(t-)\cap B_{r^\Psi}(x)))$  where
$\gamma^{\Psi}_D(t-)$ are the accepted points still alive just
before time $t$. This construction is well-defined, since there
is almost surely a $t'<t$ such that
$\gamma_D(t')=\emptyset$. For a fixed $t$, $\gamma^{\Psi}_D(t)$
yields a point process with the same distribution as $\PP_{D}$,
see e.g.\  \cite{kendall}. 

To extend this to an infinite volume process, we define the
ancestors of a point $x$ in $\gamma$ born at time $t$ by
$A(x)=\gamma(t-)\cap B_{r^\Psi}(x)$. These are all points that
(if accepted) could influence the acceptance probability of $x$.
The clan of ancestors $\mathbf{A}(x)$ consists of $x$, its
ancestors, all ancestors of ancestors etc. This is the set of all
points in the free birth-death $\gamma$ process whose acceptance
status could possibly affect the acceptance of $x$.

We first note that all ancestor clans $\mathbf{A}(x)$ are
almost surely finite. Here,  the idea of
\cite{fernandez} was to dominate the clan of ancestors by a
Galton-Watson branching process, where each split has a Poisson
distributed number of branches with mean $\lambda=\tau \kappa_d
(r^{\Psi})^d$. By standard branching process theory \cite[Chapter
1.6]{harris}, the branches die out almost
surely if $\lambda<1$, which is ensured by
\eqref{eq:tau_condition}. This ensures that the thinning
procedure above can also be applied to the full free birth-death
process $\gamma(t)$ resulting in the process $ \gamma^{\Psi}(t)$.

Next, we consider the spatial diameter of
$\mathbf{B}(x)=proj(\mathbf{A}(x))\oplus B_{r^\Psi}(o)$, where
$proj(\mathbf{A}(x))$ is the projection of the ancestor clan to
$\R^d$, such that $\mathbf{B}(x)$ is the projection to $\R^d$ of
all the balls we need to search in order to determine the
acceptance probability of $x$.  Let $Z_{n}$ be the number of
branches in the $n$th generation of the dominating branching
process. Then 
\begin{equation*}
P(diam(\mathbf{B}(x)) > 2kr^{\Psi}) \leq P(Z_k >0) \leq \E(Z_k)= \lambda^k, 
\end{equation*}
the latter equality coming from \cite[Chap. 1, Thm. 5.1]{harris}.

Similarly, define $\mathbf{B}_D(0)=\bigcup_{x\in \gamma_D(0)}
\mathbf{B}(x)$ to  the projection to $\R^d$ of all balls around
ancestors of points in $\gamma_D(t)$. Then,
\begin{equation} \label{eq:gibbs_ancestor}
P(\mathbf{B}_D(0) \cap (D\oplus B_{2kr^{\Psi}}(o))^c \neq \emptyset) \leq \E \sum_{x\in  \gamma_D(0)} \mathds{1}_{\{diam(\mathbf{B}(x))> 2kr^{\Psi}\}} \leq |D| \tau \lambda^k.
\end{equation}
This means that for any bounded $D$, $\PP_{W_n}\cap
D=\gamma_{W_n}^{\Psi}(0)\cap D$ coincides with
$\gamma^{\Psi}(0)\cap D$ with a probability that tends to 1
exponentially fast for $n\to \infty$. This implies the
convergence in the local bounded topology of $\PP_{W_n}$ (and
$\bar{\PP}_{W_n}$) to $\gamma^{\Psi}(0)$, which must thus be the
infinite volume Gibbs process.

\end{document}